\documentclass[12pt]{article}
\usepackage{authblk}
\usepackage{mathrsfs}
\usepackage{amsmath}
\usepackage{float}
\usepackage{cite}
\usepackage{esint}
\usepackage[all]{xy}
\usepackage{amssymb}
\usepackage{amsfonts}
\usepackage{amsthm}
\usepackage{color}
\usepackage{graphicx}
\usepackage{hyperref}
\usepackage{bm}
\usepackage{indentfirst}
\usepackage{geometry}

\theoremstyle{plain}
\newtheorem{thm}{Theorem}[section]
\newtheorem{lem}[thm]{Lemma}
\newtheorem{prop}[thm]{Proposition}

\newtheorem{cor}[thm]{Corollary}
\theoremstyle{definition}
\newtheorem{defn}[thm]{Definition}
\newtheorem{remark}[thm]{Remark}

\newtheorem{exmp}[thm]{Example}
\makeatletter

\newcommand{\Rmnum}[1]{\expandafter\@slowromancap\romannumeral #1@}
\makeatother

\numberwithin{equation}{section}

\begin{document}
\title{Inradius collapsed manifolds with a lower Ricci curvature bound}
\author{Zhangkai Huang\thanks{School of Mathematics, Sun Yat-sen University, China. Email: \url{huangzhk27@mail.sysu.edu.cn}}, Takao Yamaguchi\thanks{Department of Mathematics, University of Tsukuba, Japan. Email: \url{ takao@math.tsukuba.ac.jp}}}

\maketitle
\begin{abstract}
  In this paper, we study a family of $n$-dimensional Riemannian manifolds with boundary {having lower bounds on the Ricci curvatures of interior and boundary and on the second fundamental form of boundary.} A sequence of manifolds in this family is said to be inradius collapsed if their inradii tend to zero. We prove that the limit space $C_0$ of boundaries of inradius collapsed manifolds admits an isometric involution $f$, and that the limit of the manifolds themselves is isometric to the quotient space $C_0/f$. As an application, we show that the number of boundary components of inradius collapsed manifolds is at most two. Moreover, we prove that the limit space has a lower Ricci curvature bound and an upper dimension bound in a synthetic sense if in addition their boundaries are non-collapsed. 
\end{abstract}

\tableofcontents

\section{Introduction}

The convergence of Riemannian manifolds with boundary was first studied by Kodani in \cite{K90}, and later {by Anderson-Katsuda-Kurylev-Lassas-Taylor in \cite{AKKLT04} and} by Wong in \cite{W10}. In this paper, we are concerned with collapsing phenomena of {a certain family} of Riemannian manifolds with boundary.


By collapsing phenomena, we refer specifically to inradius collapsing. Here, we study Riemannian manifolds with small inradii, defined as follows.

\begin{defn}[Inradius collapsed manifolds]
Let $(M,\mathrm{g})$ be an $n$-dimensional Riemannian manifold with boundary. The \textit{inradius} of $M$ is defined as 
\[
\text{inrad}(M)=\sup_{x\in M}\mathsf{d}_\mathrm{g}(x,\partial M),
\]
 {where $\mathsf{d}_\mathrm{g}$ is the {Riemannian distance function} induced by $\mathrm{g}$.} A sequence of $n$-dimensional Riemannian manifolds {$\{(M_i,\mathrm{g}_i)\}$} is said to be \textit{inradius collapsed} if $\mathrm{inrad}(M_i)\rightarrow 0$.

\end{defn} 

The study of {manifolds with small inradii} is originated in 
{\cite{G78} and} \cite{AB98}. {In 2019}, The second author and Zhang \cite{YZ19} investigated a sequence of $n$-dimensional inradius collapsed manifolds $\{(M_i,\mathrm{g}_i)\}_{i=1}^\infty$ {{with a lower sectional
	curvature bound, two side bounds on the second fundamental forms of boundaries, i.e.}
 \begin{align}\label{jfieajifoa}
      \text{Sec}_{\mathrm{g}_i} \geqslant H,\   |\text{\Rmnum{2}}_{\partial M_i}|\leqslant \lambda.
      \end
{align}
{Wong \cite{W08} proved the precompactness of this family in the pointed Gromov-Hausdorff topology; namely, any sequence of pointed $n$-dimensional Riemannian manifolds with boundary $\{(M_i,\mathsf{d}_{\mathrm{g}_i},p_i)\}$ satisfying (\ref{jfieajifoa}) has a convergent subsequence in the pointed Gromov-Hausdorff sense.}

Under {(\ref{jfieajifoa})}, the Gauss equation implies a uniform lower sectional curvature bound $K=K(H,\lambda) $} {{on} $\partial M_i$}. Analyzing the {pointed} Gromov-Hausdorff limit space $(X,\mathsf{d},{x_0})$ of $\left\{\left(M_i,\mathsf{d}_{\mathrm{g}_i},{p_i}\right)\right\}$, {the second author and Zhang} proved that $(X,\mathsf{d})$ is an Alexandrov space with curvature bounded from below by  $K$. Moreover, for such manifolds with sufficiently small inradii, they showed that the {numbers} of components of their boundaries {are} at most 2.


In the present paper we consider the following 
family of Riemannian manifolds with boundary. {Note that the case $n=2$ is already settled in \cite{YZ19}.}
 
\begin{defn}\label{defn:moduli}
For {arbitrary} $n\in \mathbb{N}\cap [{3},\infty)$, $H,K\in\mathbb{R}$, $\lambda\geqslant 0$ and $D>0$, we denote by $\mathcal{M}(n,H,K,\lambda,D)$ the family of $n$-dimensional compact Riemannian manifolds $(M,\mathrm{g})$ with boundary such that  
      \begin{align}\label{1.1}
      \text{Ric}_\mathrm{g} \geqslant H,\ \text{Ric}_{\mathrm{g}_{\partial M}} \geqslant K,\  {\text{\Rmnum{2}}_{\partial M}\geqslant -\lambda}, \,\mathrm{diam}(M,\mathsf{d}_\mathrm{g})\leqslant D,
      \end{align}
     {where $\text{Ric}_\mathrm{g}$ {and $\text{Ric}_{\mathrm{g}_{\partial M}}$} denotes the Ricci curvature of $(M,\mathrm{g})$ {and $(\partial M,\mathrm{g}_{\partial M})$} respectively.} 
\end{defn}

Unlike the case (\ref{jfieajifoa}), the Gauss equation does not imply
a lower Ricci curvature bound of $\partial M$ even under the assumption $\text{Ric}_\mathrm{g} \geqslant H$ and $|\text{\Rmnum{2}}_{\partial M}|\leqslant \lambda$ ({see {Remark \ref{rmk2.20}} and Example \ref{ex5.6}}). This raises the question of whether the family of {boundaries of manifolds $(M,\mathrm{g})$ satisfying 
  $\text{Ric}_\mathrm{g} \geqslant H$,
$|\text{\Rmnum{2}}_{\partial M}|\leqslant \lambda$ and $\mathrm{diam}(M,\mathsf{d}_\mathrm{g})\leqslant D$
is precompact in the Gromov-Hausdorff topology}. To address this, we introduce the additional condition
$\text{Ric}_{\mathrm{g}_{\partial M}} \geqslant K$ in 
\eqref{1.1}. {With this assumption, the assumption on the second fundamental forms can be refined to $\text{\Rmnum{2}}_{\partial M}\geqslant -\lambda$.} {Moreover, by adapting Wong's proof, we demonstrate that the family $\mathcal{M}(n,H,K,\lambda,D)$ satisfies the results in \cite[Theorem 1.9]{W08} (see also Theorem \ref{thm2.3}).}


{Our first result {is} the following theorem for {limit spaces of }inradius collapsed manifolds in $\mathcal{M}(n,H,K,\lambda{,D})$.}
{\begin{thm}[Theorems \ref{thmfffff4.5} and \ref{thm4.10}]\label{aaathm1.3}
	Assume $\{(M_i,\mathrm{g}_i)\}$ is a sequence of inradius collapsed manifolds in $\mathcal{M}(n,H,K,\lambda,{D})$. Let $(C_0,\mathsf{d})$ and $(X,\mathsf{d}_X)$ be the Gromov-Hausdorff limit space{s} of $\{(\partial M_i,\mathsf{d}_{\mathrm{g}_{\partial M_i}})\}$ and $\{(M_i,\mathsf{d}_{\mathrm{g}_i})\}$ respectively. Then there exists an isometric involution $f:C_0
	\to C_0$ such that $(X,\mathsf{d}_X)$ is isometric {to} the quotient space $(C_0/f,\mathsf{d}^\ast)$, where $\mathsf{d}^\ast$ denotes the quotient metric induced by $f$.
\end{thm}
 {As an application, we get} the following result.} 
\begin{thm}\label{thm1.3}
  There exists a positive number $\epsilon=\epsilon(n,H,K,\lambda,D)$ {such} that if $(M,\mathrm{g})\in \mathcal{M}(n,H,K,\lambda,D)$ {satisfies} $\mathrm{inrad}(M)<\epsilon$, then the number of components of $\partial M$ is at most 2.
\end{thm}

In 1990's, Cheeger-Colding carried out pioneering and significant research {programs} \cite{CC96,ChCo1,ChCo2,ChCo3} on Ricci limit spaces---the Gromov-Hausdorff limit spaces of Riemannian manifolds with {a} fixed dimension and a uniform lower Ricci curvature bound. 

At the beginning of 21st century, Sturm \cite{St06a, St06b} and Lott-Villani \cite{LV09} introduced a synthetic definition of lower Ricci curvature bounds $K\in \mathbb{R}$ and upper dimension bounds $N\in [1,\infty)$ for general metric measure spaces, known as the CD$(K,N)$ condition. Building on this, Ambrosio-Mondino-Savar\'{e} \cite{AGS14a}, Gigli \cite{G13,G15},  Erbar-Kuwada-Sturm \cite{EKS15} and Ambrosio-Mondino-Savar\'{e} \cite{AMS19} introduced the RCD$(K,N)$ condition, which strengthens the CD$(K,N)$ condition by requiring the metric measure space to admit a Riemannian structure {in the sense} that its $H^{1,2}$-Sobolev space is a Hilbert space.

As an example of RCD$(K,N)$ spaces, consider weighted Riemannian manifolds. For an $n$-dimensional {closed} weighted Riemannian manifold $(M,\mathsf{d}_\mathrm{g},e^{-f}\mathrm{vol}_\mathrm{g})$ with $f\in C^\infty(M)$, the Bakry-\'{E}mery $N$-Ricci curvature tensor $\mathrm{Ric}_N$ {is} defined as
\[
\mathrm{Ric}_N:=
\left\{\begin{array}{ll}
\mathrm{Ric}_\mathrm{g}+\mathrm{Hess}_\mathrm{g}(f)-\frac{df\otimes df}{N-n}&\text{if}\ N>n,\\
\mathrm{Ric}_\mathrm{g}& \text{if $N=n$ and $f$ is a constant},\\
-\infty&\text{otherwise}.
\end{array}\right.
\]
It is known that if $\mathrm{Ric}_N\geqslant K$, then $(M,\mathsf{d}_\mathrm{g},e^{-f}\mathrm{vol}_\mathrm{g})$ is an RCD$(K,N)$ space.

{Recently, N{\'u}{\~n}ez-Zimbr{\'o}n-Pasqualetto-Soultanis proved in \cite[Corollary 1.4]{NPS25} that for every RCD$(K,N)$ space $(X,\mathsf{d},\mathfrak{m})$, the metric space $(X,\mathsf{d})$ is universally infinitesimally Hilbertian in the following sense.
\begin{defn}
	A complete metric space $(X,\mathsf{d})$ is said to be \textit{universally infinitesimally Hilbertian} if $H^{1,2}({ \mathrm{supp}(\mu)},\mathsf{d},\mu)$ is a Hilbert space for any Radon measure $\mu$ on $X$, { where $\mathrm{supp}(\mu)$ is the support of $\mu$.}
	\end{defn}

We summarize the geometric properties of the limit space of inradius collapsed manifolds in $\mathcal{M}(n,H,K,\lambda,D)$ as follows.

\begin{thm}[Theorems \ref{thm4.10} and \ref{thm4.30}]\label{thm1.6}
	Under the assumption of Theorem \ref{aaathm1.3}, the metric space $(X,\mathsf{d}_X)$ is geodesically non-branching and universally infinitesimally Hilbertian.
\end{thm}

}

As a synthetic counterpart of non-collapsed Ricci limit spaces, De Philippis-Gigli \cite{DG18} proposed the definition of non-collapsed RCD$(K,N)$ spaces. To be precise, an RCD$(K,N)$ space is called non-collapsed when its reference measure is the $N$-dimensional Hausdorff measure $\mathscr{H}^N$. Such spaces bear better regularity properties than general RCD$(K,N)$ spaces. See for instance \cite{ABS19,KM21}.

Han \cite{H20} (see also Theorem \ref{thm2.15}) proved that manifolds in $\mathcal{M}(n,H,K,\lambda,{D})$ do not satisfy the RCD$(H,n)$ condition in general. Our second main result is stated as follows, which says that the Gromov-Hausdorff limit space of inradius collapsed manifolds in $\mathcal{M}(n,H,K,\lambda,{D})$ is a {{\it non-collapsed}} RCD$(K,n-1)$ space when their boundaries are non-collapsed.
\begin{thm}\label{thm1.4}
  Assume $\{(M_i,\mathrm{g}_i)\}$ is a sequence of inradius collapsed manifolds in $\mathcal{M}(n,H,K,\lambda,D)$ such that 
\begin{align}\label{1.2}
{\inf_i}\, \mathrm{vol}_{\mathrm{g}_{\partial M_i}}({\partial M_i})\,{>0}.
\end{align}
If $\{(M_i,\mathsf{d}_{\mathrm{g}_i})\}$ converges to a metric space $(X,\mathsf{d}_X)$ in the Gromov-Hausdorff sense, then $(X,\mathsf{d}_X,\mathscr{H}^{n-1})$ is a non-collapsed $\mathrm{RCD}(K,n-1)$ space.
\end{thm}
{\begin{remark}
		When {$\lambda=0$, a result of Han  \cite{H20} implies that for any inradius collapsed manifolds $\{(M_i,\mathrm{g}_i)\}$ in $\mathcal{M}(n,H,K,\lambda,D)$, after passing to a subsequence,} the measured Gromov-Hausdorff limit of {$(M_i,\mathsf{d}_{\mathrm{g}_i},(\mathrm{vol}_{\mathrm{g}_i}(M_i))^{-1}\mathrm{vol}_{\mathrm{g}_i})$ (where $(\mathrm{vol}_{\mathrm{g}_i}(M_i))^{-1}\mathrm{vol}_{\mathrm{g}_i}$ is the renormalized measure)} is an RCD$(H,n)$ space. {However, even under (\ref{1.2}), the limit measure may not coincide with the $(n-1)$-dimensional Hausdorff measure (up to a multiplicative constant). See Example \ref{exmp5.2}.

		For the case $\lambda>0$, such mGH limit space may fail to satisfy any RCD condition altogether, as shown in Example \ref{eexmp5.4}.}

\end{remark}}
 \begin{remark}
		
		{Under the assumption of Theorem \ref{aaathm1.3}, there exists a Randon measure $\mu$ on $X$ such that $\mathrm{supp}(\mu)=X$ and $(X,\mathsf{d}_X,\mu)$ is an RCD$(K,n-1)$ space, except in the following cases:
	\begin{itemize}
		\item[$(1)$] $\lambda>0$, $\lim_{i\to\infty}\mathrm{vol}_{\mathrm{g}_{\partial M_i}}({\partial M_i})=0$, $C_0$ is connected and $f$ has no fixed point.
		\item[$(2)$] $\lambda>0$, $\lim_{i\to\infty}\mathrm{vol}_{\mathrm{g}_{\partial M_i}}({\partial M_i})=0$, $f$ has a fixed point, and the complement of the fixed point set of $f$ is non-empty and connected. 
	\end{itemize} See Remark \ref{rmk4.33}.
	}

		{In general, it remains {unknown}  whether {the limit space $(X,\mathsf{d}_X)$ obtained in Theorem \ref{aaathm1.3}} admits a full-support Radon measure such that the resulting metric measure space satisfies some RCD condition.}

\end{remark}

{As a noncompact analogue, we study inradius collpased manifolds in the following family.
\begin{defn}
	For $n\in \mathbb{N}\cap [{3},\infty)$, $H,K\in\mathbb{R}$, $\lambda\geqslant 0$, let $\mathcal{M}(n,H,K,\lambda)$ be the family of $n$-dimensional Riemannian manifolds $(M,\mathrm{g})$ with boundary such that  
	\[
	\text{Ric}_\mathrm{g} \geqslant H,\ \text{Ric}_{\mathrm{g}_{\partial M}} \geqslant K,\  {\text{\Rmnum{2}}_{\partial M}\geqslant -\lambda}.
	\]
\end{defn}

In connection with Theorems \ref{aaathm1.3}, \ref{thm1.3}, \ref{thm1.6} and \ref{thm1.4}, we summarize our findings as follows. This result confirms that the aforementioned theorems are indeed independent of the parameter $D$. 
 \begin{thm}\label{thm1.10}
		Assume $\{(M_i,\mathrm{g}_i,p_i)\}$ is a sequence of inradius collapsed manifolds in $\mathcal{M}(n,H,K,\lambda)$ with $p_i\in\partial M_i$. Let $(C_0,\mathsf{d},p)$ and $(X,\mathsf{d}_X,x)$ be the pointed Gromov-Hausdorff limit space{s} of $\{(\partial M_i,\mathsf{d}_{\mathrm{g}_{\partial M_i}},p_i)\}$ and $\{(M_i,\mathsf{d}_{\mathrm{g}_i},p_i)\}$ respectively. Then the following holds.
	\begin{itemize}
		\item[$(1)$] {There} exists an isometric involution $f:C_0
		\to C_0$ such that $(X,\mathsf{d}_X,x)$ is isometric {to} the quotient space $(C_0/f,\mathsf{d}^\ast,p)$, where $\mathsf{d}^\ast$ denotes the quotient metric induced by $f$. 
		\item[$(2)$] The metric space $(X,\mathsf{d}_X)$ is geodesically non-branching and universally infinitesimally Hilbertian.
		\item[$(3)$] If in addition \[
			{\inf_i}\, \mathrm{vol}_{\mathrm{g}_{\partial M_i}}(B_1^{{ \partial M_i}}(p_i))>0,
		\]
		then $(X,\mathsf{d}_X,\mathscr{H}^{n-1})$ is a non-collapsed $\mathrm{RCD}(K,n-1)$ space.
		
	\end{itemize}

\end{thm}

\begin{thm}\label{thm1.11}
	 There exists a positive number $\epsilon=\epsilon(n,H,K,\lambda)$ such that if $(M,\mathrm{g})\in \mathcal{M}(n,H,K,\lambda)$ {satisfies} $\mathrm{inrad}(M)<\epsilon$, then the number of components of $\partial M$ is at most 2.
	\end{thm}}

{We here explain the main idea of the proofs of our main theorems and the organization of this paper. The overall approach of this paper follows that of \cite{YZ19}. However, in the detailed proofs, we generalize the arguments in \cite{YZ19} that relied heavily on the assumption of a lower bound on the sectional curvature.}

In {S}ection \ref{sec2}, we begin {by presenting fundamental concepts concerning Ricci limit spaces and RCD spaces. We then} recall Wong's extension procedure \cite{W08} for manifolds in $\mathcal{M}(n,H,K,\lambda,D)$, {which involves attaching warped cylinders along the manifolds' boundaries.}

In {S}ection \ref{sec3}, we examine a sequence of inradius collapsed manifolds $\{(M_i,\mathrm{g}_i)\}$ in $\mathcal{M}(n,H,K,\lambda,D)$ and its {Gromov-Hausdorff} limit space $(X,\mathsf{d}_X)$. This space coincides with the limit space of the boundaries equipped with exterior metrics {$\{(\partial M_i,{\mathsf{d}_{\mathrm{g}_i}})\}$.} Simultaneously, the sequence of the boundaries with intrinsic metrics $\{(\partial M_i,\mathsf{d}_{\mathrm{g}_{\partial M_i}})\}$, is precompact in the Gromov-Hausdorff topology, converging to a limit space $(C_0,\mathsf{d})$. Moreover, the identity maps {$\mathrm{id}:(\partial M_i,\mathsf{d}_{\mathrm{g}_{\partial M_i}})\to(\partial M_i,{\mathsf{d}_{\mathrm{g}_i})}$} converge to a 1-Lipschitz surjection $\eta_0:(C_0,\mathsf{d})\rightarrow (X,\mathsf{d}_X)$. We show that ${\eta_0}^{-1}{(x)}$ contains at most two points for any $x\in X$. This naturally induces an involution $f:C_0\rightarrow C_0$ (see Definition \ref{11111defn3.9} and (\ref{feaohfoiaehfiohae})).

In {S}ection \ref{sec4}, we {take a point $p\in C_0$ and let $x=\eta_0(p)$.} We study the tangent spaces $T_p C_0$ and $T_x X$ and the induced tangent map $\eta_\infty:T_p C_0\rightarrow T_x X$. We prove $\eta_\infty$ preserves the distance from the origin, demonstrating that $f$ is an isometric involution. Consequently, $\eta_0$ becomes a quotient map under the $\mathbb{Z}_2$-action $\{f,\mathrm{id}\}$, from which Theorems {\ref{aaathm1.3} and \ref{thm1.3} follow. Combining Theorem \ref{aaathm1.3} with \cite{D20} and \cite{GMMS18}, we derive Theorems \ref{thm1.6} and \ref{thm1.4} respectively.} {Note that the application of the Busemann function plays an important role in this section. {Finally, we extend our analysis to the noncompact case of Ricci inradius collapsed manifolds. By employing a methodology similar to that in \cite{YZ19}, we likewise obtain results analogous to those presented in Theorems \ref{aaathm1.3}, \ref{thm1.3}, \ref{thm1.6} and \ref{thm1.4}.}}

In {S}ection \ref{section5}, we present characteristic examples of inradius collapsed manifolds. {As established in \cite[Theorem 1.6]{YZ19}, for any manifold with boundary $(M,\mathrm{g})$ satisfying (\ref{jfieajifoa}), if the inradius of $M$ is sufficiently small {and} $\partial M$ has two components, {then} $M$ is diffeomorphic to $W\times[0,1]$, where $W$ is a component of $\partial M$. Nevertheless, using Menguy's construction \cite{M00},} we exhibit {a sequence of} {four dimensional} Ricci inradius collapsed manifolds {such that the boundary of each manifold has two components, both diffeomorphic to $\mathbb{S}^3$, yet the second Betti number of these manifolds tends to infinity.}




\textbf{Acknowledgement.} The first author acknowledges the support of Fundamental Research Funds for the Central Universities, Sun Yat-sen University, Grant Number 24qnpy105. The Second author was supported by JSPS KAKENHI  Grant Number
21H00977.

\section{Preliminaries}\label{sec2}


Throughout this paper, we denote by $J=J(c_1,\ldots,c_n)$ a positive constant depending on $c_1,\ldots,c_n$, which may vary from line to line.

{We also make some convention on complete and separable metric spaces $( X,\mathsf{d})$. For a {rectifiable} curve $\gamma$ {in $(X,\mathsf{d})$}, its length is denoted by $L^{(X,\mathsf{d})}(\gamma)$ (or briefly $L^\mathsf{d}(\gamma)$, $L(\gamma)$). The set of Lipschitz functions is denoted by $ \text{Lip}( X,\mathsf{d})$. For $f\in \text{Lip}( X,\mathsf{d})$, the local Lipschitz constant of $f$ at $x\in  X$ is defined  as} 
\[
\text{lip}\
 f(x)=\left\{\begin{aligned}
\limsup\limits_{y\rightarrow x} \frac{|f(y)-f(x)|}{\mathsf{d}(y,x)}&\ \ \ \text{if $x\in  X$ is not isolated},\\
0\ \ \ \ \ \ \ \ \ \ \ \ \ \ \ \ \ \ \ \ \ \ \ \ \ &\ \ \ \text{otherwise}.
\end{aligned}\right.
\]{The open ball of radius $R$ centered at $x$ is denoted as {$B_R^{(X,\mathsf{d})}(x)$} (or briefly {$B_R^X(x)$}, $B_R(x)$). If $ (X,\mathsf{d})$ is geodesic, a unit speed minimal geodesic from $x$ to $y$ is denoted by $\gamma^{(X,\mathsf{d})}_{x,y}$ (or briefly $\gamma^\mathsf{d}_{x,y}$, $\gamma_{x,y}$). We also {denote by} $\mathscr{H}^m_\mathsf{d}$ ({or briefly $\mathscr{H}^m$}) the $m$-dimensional Hausdorff measure on $(X,\mathsf{d})$ for $m\in [1,\infty)$.}



\subsection{Ricci limit spaces}
This subsection is aimed at recalling some properties of Ricci limit spaces. The precise definition of Gromov-Hausdorff convergence (GH convergence), {Gromov-Hausdorff distance $\mathsf{d}_{\mathrm{GH}}$}, pointed Gromov-Hausdorff convergence (pGH convergence), {pointed Gromov-Hausdorff distance $\mathsf{d}_{\mathrm{pGH}}$}, measured Gromov-Hausdorff convergence (mGH convergence) and pointed measured Gromov-Hausdorff convergence (pmGH convergence) are omitted. For detailed treatments, see \cite{F87} and \cite{GMS15}.

Fix $n\in \mathbb{N}_+$, $D>0$. Let $\{(M_i,\mathrm{g}_i)\}$ be a sequence of $n$-dimensional closed Riemannian manifolds with 
\[
\mathrm{Ric}_{\mathrm{g}_i}\geqslant {-1}, \ \mathrm{diam}(M_i,\mathsf{d}_{\mathrm{g}_i})\leqslant D.
\]
{Owing to Gromov's compactness theorem, after passing to a subsequence, {we may assume that} there exists a complete and separable metric measure space $ (X,\mathsf{d},\mathfrak{m})$ (which is indeed an RCD$(-1,n)$ metric measure space by Section \ref{sec2.3}) such that}
\[
\left(M_i,\mathsf{d}_{\mathrm{g}_i},\frac{\mathrm{vol}_{\mathrm{g}_i}}{\mathrm{vol}_{\mathrm{g}_i}(M_i)}\right)\xrightarrow{\mathrm{mGH}}(X,\mathsf{d},\mathfrak{m}).
\]



We next introduce an important result on regular sets, whose precise definition is given as follows.
\begin{defn}[k-regular sets]
A point $x\in X$ is said to be a \textit{$k$-regular} point if for any sequence of positive numbers $\{s_i\}$ with $s_i\rightarrow 0$ it holds that
\[
({s_i}^{-1}X,x)\xrightarrow{\mathrm{pGH}}(\mathbb{R}^k,0_k),
\]
{where we denote the rescaled pointed metric space $\left(X,{s_i}^{-1}\mathsf{d},x\right)$ by $({s_i}^{-1}X,x)$ for simplicity.} The set of all $k$-regular points of $X$ is denoted by $\mathcal{R}_k(X)$.
\end{defn}

The following theorem on {the \textit{dimension}} of Ricci limit spaces is derived from \cite[Theorem 0.1]{C97}, \cite[Theorem 3.1]{ChCo1} and \cite[Theorem 1.18]{CN12}.
\begin{thm}\label{thm2.8}
There exists a unique {integer} $k\leqslant n$ such that $\mathfrak{m}(X\setminus \mathcal{R}_k(X))=0$. This $k$ is called the dimension of $X$.
\end{thm}

{If in addition $\inf_i \mathrm{vol}_{\mathrm{g}_i}(M_i)>0$, then $(X,\mathsf{d},\mathfrak{m})$ is said to be a non-collapsed Ricci limit space. In this case, the following theorem is known due to \cite{ChCo1}.}
\begin{thm}\label{ncRiccilimit}
{For such $(X,\mathsf{d},\mathfrak{m})$, {both the dimension and Hausdorff dimension are} $n$ and the reference measure satisfies $\mathfrak{m}=(\mathscr{H}^n(X))^{-1}\mathscr{H}^n$. {Moreover, the Hausdorff dimension of $X\setminus \mathcal{R}_n(X)$ does not exceed $n-2$.}

}
\end{thm}

Finally, let us recall the following Cheeger-Colding's {splitting theorem.}
\begin{thm}[\cite{CC96}]\label{CCsplit}
Let $\{(N_i,\mathrm{h}_i)\}$ be a sequence of $n$-dimensional Riemannian manifolds such that $\mathrm{Ric}_{\mathrm{h}_i}\geqslant -\delta_i$  $($where $\delta_i\rightarrow 0)$ and that $\{(N_i,\mathsf{d}_{\mathrm{h}_i},q_i)\}$ $\mathrm{pGH}$ converges to a pointed metric space $(Y,\mathsf{d}_Y,y)$. If $(Y,\mathsf{d}_Y)$ contains a line $\gamma$, then there exists a length space {$(\Omega,\mathsf{d}_\Omega)$ such that $(Y,\mathsf{d}_Y)$ is isometric to the product metric space $\left(\Omega\times \mathbb{R},\sqrt{{\mathsf{d}_\Omega}^2+{\mathsf{d}_\mathbb{R}}^2}\right)$. Moreover, if we denote the projection onto the first coordinate by $\pi$, then $\pi(\gamma)$ is a single point $\omega\in\Omega$.}
\end{thm}
\subsection{RCD spaces}\label{sec2.3}

In this paper, when we say $(X,\mathsf{d},\mathfrak{m})$ is a metric measure space, we mean that $(X,\mathsf{d})$ is a complete separable metric space and $\mathfrak{m}$ is a nonnegative Borel measure which is finite on any bounded subset of $X$ such that $\text{supp}(\mathfrak{m})=X$. 

Let $(X,\mathsf{d},\mathfrak{m})$ be a metric measure space.
\begin{defn}[Cheeger energy]
   The \textit{Cheeger energy} of $h\in L^2(\mathfrak{m})$ is defined as
  \[
  \text{Ch}(h):=\inf_{\{h_n\}} \left\{\liminf_{n \rightarrow \infty} \frac{1}{2}\int_X (\text{lip} \mathop{h_n})^2\, \mathrm{d}\mathfrak{m} \right\},
  \]
  where the infimum is taken among all sequences $\{h_n\}\subset \text{Lip}(X,\mathsf{d})\cap L^2(\mathfrak{m})$ with $\|h_n-h\|_{L^2(\mathfrak{m})}\rightarrow 0$. 

\end{defn}

\begin{remark}\label{rmk2.6}
The Sobolev space $H^{1,2}(X,\mathsf{d},\mathfrak{m})$ defined as the set of all $L^2(\mathfrak{m})$-integrable function with finite Cheeger energy. {For every $h\in H^{1,2}(X,\mathsf{d},\mathfrak{m})$, by \cite{ACD15} and\cite[Section 8.3]{AGS13}}, there exists $\{h_n\}\subset \text{Lip}(X,\mathsf{d})\cap L^2(\mathfrak{m})$ such that $h_n\to h$ in $L^2(\mathfrak{m})$, that {$\mathrm{lip}\, h_n\to \exists |\nabla h|$ in $L^2(\mathfrak{m})$ and that
\[
\mathrm{Ch}(h)=\frac{1}{2}\int_X |\nabla h|^2\, \mathrm{d}\mathfrak{m}.
\]}This $|\nabla h|$ is unique and is called the \textit{minimal relaxed slope} of $h$.

\end{remark}

\begin{remark}\label{rmk2.12}
$(X,\mathsf{d},\mathfrak{m})$ is said to be {\it infinitesimally Hilbertian} if $H^{1,2}(X,\mathsf{d},\mathfrak{m})$ is a Hilbert space. In this case, for any $h_i\in H^{1,2}(X,\mathsf{d},\mathfrak{m})$ $(i=1,2)$, according to \cite{AGS14a}, the following $L^1(\mathfrak{m})$-integrable function is well-defined:
 \[
  \langle \nabla h_1, \nabla h_2\rangle :=  \lim_{\epsilon \rightarrow 0}\frac{|\nabla(h_1+\epsilon h_2)|^2-|\nabla h_1|^2}{2\epsilon}.
  \]

\end{remark}

\begin{defn}[The Laplacian]
Assume $(X,\mathsf{d},\mathfrak{m})$ is infinitesimally Hilbertian. The domain of Laplacian, namely $D(\Delta)$, is defined as the set of all $h\in H^{1,2}(X,\mathsf{d},\mathfrak{m})$ such that the following holds for some $\xi \in L^2(\mathfrak{m})$. 

  \[
  \int_X \langle \nabla h, \nabla \varphi\rangle\, \mathrm{d}\mathfrak{m}= - \int_X \varphi\xi\, \mathrm{d}\mathfrak{m},\ \ \forall \varphi\in H^{1,2}(X,\mathsf{d},\mathfrak{m}).
  \]
We denote by $\Delta h:=\xi$ because $\xi$ is unique if it exists.

\end{defn}

We are now in a position to introduce the definition of RCD$(K,N)$ space as follows. See \cite{AGS15,AMS19,EKS15,G15}.

\begin{defn}\label{defn2.9}
  Let $K\in \mathbb{R}$ and $N\in [1,\infty)$. $(X,\mathsf{d},\mathfrak{m})$ is said to be an RCD$(K,N)$ space if
\begin{enumerate}
  \item[$(1)$] $(X,\mathsf{d},\mathfrak{m})$ is infinitesimally Hilbertian.

  \item[$(2)$] There exist $ x \in X$ and $C >0$, such that for any $r>0$ it holds that $\mathfrak{m} (B_r(x)) \leqslant C e^{Cr^2}$.

  \item[$(3)$] Any $h \in H^{1,2}(X,\mathsf{d},\mathfrak{m})$ satisfying $|\nabla h|\leqslant 1$ $\mathfrak{m}$-a.e. has a 1-Lipschitz representative.

  \item[$(4)$]  For any $ h\in D(\Delta)$ with $\Delta h \in H^{1,2}(X,\mathsf{d},\mathfrak{m})$, the following holds for any $\varphi \in \mathrm{Test}F\left(X,\mathsf{d},\mathfrak{m}\right)$ with
  $ \varphi \geqslant 0$,
\[
  \frac{1}{2}\int_X |\nabla h|^2 \Delta \varphi\, \mathrm{d}\mathfrak{m} 
  \geqslant \int_X \varphi \left(\langle \nabla h , \nabla \Delta h \rangle +K |\nabla h|^2   + \frac{(\Delta h)^2}{N}  \right) \mathrm{d}\mathfrak{m},
  \]
where $\mathrm{Test}F(X,\mathsf{d},\mathfrak{m})$ is the class of test functions defined by
\end{enumerate}
 \[
\mathrm{Test}F(X,\mathsf{d},\mathfrak{m}):=\left\{h\in \text{Lip}(X,\mathsf{d})\cap D(\Delta)\cap L^\infty(\mathfrak{m}):\Delta h\in H^{1,2}(X,\mathsf{d},\mathfrak{m})\cap L^\infty(\mathfrak{m})\right\}.
\]

\end{defn}
\begin{remark}
If in addition $\mathfrak{m}=\mathscr{H}^N$, then $(X,\mathsf{d},\mathfrak{m})$ is said to be a non-collapsed RCD$(K,N)$ space. In this case, it is readily checked by \cite{DG18} that $N$ is an integer.
\end{remark}
{\begin{remark}\label{rmk2.11}
Let $(X,\mathsf{d},\mathfrak{m})$ be an RCD$(K,N)$ space. Then it admits a local Poincar\'e inequality (\cite[Theorem 1]{R12}) and the reference measure $\mathfrak{m}$ satisfies the doubling condition (\cite[Theorem 5.31]{LV09}, \cite[Theorem 2.3]{St06b}). As a consequence, \cite[Theorem 6.1]{Ch99} implies that for every $g\in \mathrm{Lip}(X,\mathsf{d})\cap H^{1,2}(X,\mathsf{d},\mathfrak{m})$ we have $|\nabla g|=\mathrm{lip}\,g$ $\mathfrak{m}$-a.e.
\end{remark}
}
\begin{thm}[RCD condition for manifolds with boundary {\cite[Theorem 1.1]{H20}}]\label{thm2.15}
Let $(M,\mathrm{g})$ be an $n$-dimensional Riemannian manifold with boundary. {Then} {the metric measure space $(M,\mathsf{d}_\mathrm{g},\mathrm{vol}_\mathrm{g})$} is an $\mathrm{RCD}(K,n)$ space if and only if $M$ is path-connected and its Ricci curvature, second fundamental form satisfy 
\[
\mathrm{Ric}_\mathrm{g}\geqslant K,\ \  \mathrm{\Rmnum{2}}_{\partial M}\geqslant 0.
\] 
\end{thm}

We introduce the following two important results about RCD$(K,N)$ spaces.

\begin{thm}[Precompactness of pointed RCD$(K,N)$ spaces under pmGH convergence \cite{GMS15}]\label{thm2.16}
Let $\left\{(X_i,\mathsf{d}_i,\mathfrak{m}_i,x_i)\right\}$ be a sequence of pointed $\mathrm{RCD}(K,N)$ spaces such that
\[
0<\liminf\limits_{i\rightarrow \infty} \mathfrak{m}_i\left(B_1^{X_i}(x_i)\right)\leqslant\limsup\limits_{i\rightarrow \infty} \mathfrak{m}_i\left(B_1^{X_i}(x_i)\right)<\infty.
\]
Then there exists a subsequence $\left\{\left(X_{i(j)},\mathsf{d}_{i(j)},\mathfrak{m}_{i(j)},x_{i(j)}\right)\right\}$ which $\mathrm{pmGH}$ converges to a pointed $\mathrm{RCD}(K,N)$ space $(X,\mathsf{d},\mathfrak{m},x)$.  
\end{thm}

\begin{thm}[RCD$(K,N)$ spaces are non-branching \cite{D20}]\label{thm2.17}
Suppose $(X,\mathsf{d},\mathfrak{m})$ is an $\mathrm{RCD}(K,N)$ space. Then {the metric space $(X,\mathsf{d})$ is geodesically non-branching. That is,} given any two constant speed geodesics $\gamma_i$ $(i=1,2)$, which are parameterized on the unit interval, if there exists $t\in (0,1)$, such that 
\[
\gamma_1(s)=\gamma_2(s),\ \forall s\in [0,t],
\]
then {we have}
\[
\gamma_1(s)=\gamma_2(s),\ \forall s\in [0,1].
\]
\end{thm}

Let us end this subsection by recalling a result on quotient spaces of metric measure spaces.

Assume $G$ is a compact Lie group acting isometrically on $(X,\mathsf{d})$. That is, for any $g\in G$, the map
\[
\begin{aligned}
\tau_g:X &\longrightarrow X\\
      x &\longmapsto gx,
\end{aligned}
\]
 is an isometry. Let us set $X^\ast:=X/ G$ as the quotient space and $p: X\rightarrow X^{\ast}$ as the quotient map. The quotient metric $\mathsf{d}^\ast$ is defined as
 \begin{align}\label{a2.3}
 	\mathsf{d}^\ast(x^\ast,y^\ast):=\inf\limits_{x\in p^{-1}(x^\ast),\ y\in p^{-1}(y^\ast)}\mathsf{d}(x,y).
 \end{align}

 {\begin{prop}[Quotient spaces of non-branching metric spaces are non-branching]\label{prop2.15}
 	If $(X,\mathsf{d})$ is geodesically non-branching, then $(X^\ast,\mathsf{d}^\ast)$ is also geodesically non-branching.
 \end{prop}} 
\begin{proof}
	We prove by contradiction. Without loss of generality, assume there exists two distinct unit speed minimal geodesics $\gamma_{i}:[-1,1]\to X^\ast$ ($i=1,2$) such that $\gamma_1=\gamma_2$ on $[-1,0]$, but $\gamma_1(1)\neq\gamma_2(1)$.
	
	Let $x=\gamma_1(-1)$, $y=\gamma_1(0)$ and $z_i=\gamma_i(1)$. Fix a point $\tilde{y}\in p^{-1}(y)$. Because $G$ is compact, there exists $\tilde{x}\in p^{-1}(x)$ and $\tilde{z}_i\in p^{-1}(z_i)$ such that
	\[
	\mathsf{d}(\tilde{x},\tilde{y})=\mathsf{d}(\tilde{y},\tilde{z}_1)=\mathsf{d}(\tilde{y},\tilde{z}_2)=1.
	\]
This together with (\ref{a2.3}) implies $\mathsf{d}(\tilde{x},\tilde{z}_1)=\mathsf{d}(\tilde{x},\tilde{z}_2)=2$. 

Let $\tilde{\sigma}:[-1,0]\to X$ and $\tilde{\sigma}_i:[0,1]\to X$ be unit speed minimal geodesics from $\tilde{x}$ to $\tilde{y}$ and from $\tilde{y}$ to $\tilde{z}_i$ respectively.

Then the geodesic $\gamma_i$ defined as
\[
\gamma_i(t)=\left\{\begin{aligned}
	\tilde{\sigma}(t),\ &\ t\in [-1,0];\\
	\tilde{\sigma_i}(t),&\ t\in (0,1],
\end{aligned}\right.
\]
is a unit speed minimal geodesic from $\tilde{x}$ to $\tilde{z}_i$. This contradicts Theorem \ref{thm2.17}.
\end{proof}
If moreover $G$ is measure-preserving in the following sense:
\[
 (\tau_g)_\sharp \mathfrak{m}=\mathfrak{m},\ \forall g\in G,
\]
where $(\tau_g)_\sharp {\mathfrak{m}}$ means the push forward measure under $\tau_g$, then we can endow $X^\ast$ with the quotient measure $\mathfrak{m}^\ast:= p_\sharp \mathfrak{m}$.

\begin{thm}[\cite{GMMS18}]\label{thm2.18}
If $(X,\mathsf{d},\mathfrak{m})$ is an $\mathrm{RCD}(K,N)$ space, then $(X^\ast,\mathsf{d}^\ast,\mathfrak{m}^\ast)$ is also an $\mathrm{RCD}(K,N)$ space.
\end{thm}

\subsection{Manifolds with boundary: gluing and extension}\label{sec2.2}

This subsection is dedicated to recalling the extension procedure in \cite{W08}.

  Take $(M,\mathrm{g})\in \mathcal{M}(n,H,K,\lambda,D)$, and consider the product space $C_M:= \partial M \times [0,t_0]$ with the Riemannian metric $\mathrm{g}_{C_M}=\mathrm{d}t^2+\phi^2(t) \mathrm{g}_{\partial M}$, where
\[
\phi(t)=(1-\varepsilon_0)\exp\left(-\frac{2\lambda t_0^2}{1-\varepsilon_0}\left(\frac{1}{t_0-t}-\frac{1}{t_0}\right)\right)+\varepsilon_0,
\]
for some fixed $\varepsilon_0\in (0,1)$, $t_0>0$. Then direct computation gives
{ \[
\left\{
\begin{aligned}
  \phi(0)=1,\  \phi(t_0)=\varepsilon_0,\ \ \ \ \ \ \ \ \ \ \ \ \ \ \ \ \ \ \ \ \ \ \ \ \ \ \ \\
  \phi'(0)=-2\lambda,\  \phi'(t_0)=0,\ {\phi'<0\text{ on } (0,t_0)}, \\
  {\max_{[0,t_0]}}\ |\phi''/\phi| \leqslant J({\lambda, \varepsilon_0,t_0}).\ \ \ \ \ \ \ \ \ \ \ \ \ \ \ \ \ \ \ \ \  
\end{aligned}
\right.
\]
Moreover, we have $\text{Ric}_{\mathrm{g}_{C_M}}\geqslant -J(n,\lambda,K,\varepsilon_0,t_0)$.
}

For any $t\in [0,t_0]$, let $C_M^t:=\partial M \times \{t\}$ and  
$\mathrm{\Rmnum{2}}^{C_M^t}$  be the second fundamental form of $C_M^t$ defined as:
\[
\mathrm{\Rmnum{2}}^{C_M^t}(V,W)=-\left\langle\nabla_V\frac{\partial }{\partial t}, W\right\rangle,\ \ \forall V,W\in T(C_M^t).
\]

It is obvious that 
\begin{equation*}
\mathrm{\Rmnum{2}}^{C_M^0}+\mathrm{\Rmnum{2}}^{\partial M}\geqslant {\lambda}>0.
\end{equation*}
Therefore we are able to apply the following Perelman's gluing theorem.

\begin{thm}[\cite{P97}]\label{thm2.1}
  Let $(N_i,\mathrm{h}_i)$ $(i=1,2)$ be two $n$-dimensional compact Riemannian manifolds with boundary. Assume  $\phi: \partial N_1 \rightarrow \partial N_2$ is an isometry. Assume $\mathrm{Ric}_{\mathrm{h}_i}>H$ for some $H\in\mathbb{R}$, and the sum of the second fundamental forms $\mathrm{\Rmnum{2}}_{\partial N_1}|_p+\mathrm{\Rmnum{2}}_{\partial N_2}|_{\phi(p)}$, with $T_p N_1$ and $T_{\phi(p)} N_2$ identified by $d\phi$, is positive definite for every point $p \in \partial N_1$. Then the induced metric on $N=N_1 \cup_\phi N_2$ can be $C^2$-smoothed out in an arbitrarily small tubular neighborhood of $\partial N_1$ in $N$ to have a strict lower Ricci curvature bound $H$.
\end{thm}

\begin{remark}\label{rmk2.2}
Let $N=N_1 \bigcup_\phi N_2$ be the space constructed as above. For any $x,y \in N$, the distance $\mathsf{d}(x,y)$ is defined as the infimum of the length of all piecewise smooth curves connecting $x,y$. This makes $(N,\mathsf{d})$ a complete metric space. According to \cite[Lemma 2]{BWW19}, for any $\epsilon>0$, there exists a $C^2$ Riemannian metric $\tilde{\mathrm{h}}$ on $N$ such that
\begin{enumerate}

\item[$(1)$] $\text{Ric}_{\tilde{\mathrm{h}}}>H$.

\item[$(2)$] $\mathsf{d}_{\mathrm{GH}}((N, \mathsf{d}), (N, \mathsf{d}_{\tilde{\mathrm{h}}}))<\epsilon$, where $\mathsf{d}_{\mathrm{GH}}$ is the Gromov-Hausdorff distance.
\end{enumerate}
\end{remark}

The main technique of \cite{W08} is to apply Theorem \ref{thm2.1} to the glued space $\widetilde{M}:= M \bigcup_{\mathrm{id}} C_M$. We collect two useful results by Wong as follows.
 
\begin{thm}[{\cite[Theorem 1.9]{W08}}]\label{thm2.3}
For all {$n\in \mathbb{N}\cap [3,\infty)$}, $H,K\in\mathbb{R}$, $\lambda\geqslant 0$ and $D>0$, there exist constants {$J_i=J_i(n,H,K,\lambda,D)$} $(i=1,2)$ such that

\begin{enumerate}
\item[$(1)$] For any $(M,\mathrm{g})\in \mathcal{M}(n,H,K,\lambda,D)$, any boundary component $\partial M^\alpha$ has {an} intrinsic diameter bound
  \[
  \mathrm{diam}(\partial M^\alpha,\mathsf{d}_{\mathrm{g}_{\partial M}})\leqslant J_1.
  \]

\item[$(2)$] For any $(M,\mathrm{g})\in \mathcal{M}(n,H,K,\lambda,D)$, $\partial M$ has at most $J_2$ components.

\item[$(3)$]   $\mathcal{M}(n,H,K,\lambda,D)$ is precompact in the Gromov-Hausdorff topology. 
\end{enumerate}
\end{thm}

\begin{remark}\label{rmk2.20}
Let us take an $n$-dimensional Riemannian manifold with boundary $(M,\mathrm{g})$ and a point $p\in \partial M$. Let us also take an orthonormal basis $\{e_i\}_{i=1}^n$ {in} $T_p M$ such that $e_n$ is perpendicular to $T_p \partial M$.

Then according to the Gauss equation one has that at $p$,
\[
\mathrm{R}_{\mathrm{g}_{\partial M} }(e_i,e_j,e_k,e_l)={\mathrm{R}}_\mathrm{g}(e_i,e_j,e_k,e_l)+\mathrm{\Rmnum{2}}(e_i,e_l)\mathrm{\Rmnum{2}}(e_j,e_k)-{ \mathrm{\Rmnum{2}}(e_j,e_l)\mathrm{\Rmnum{2}}(e_i,e_k)},
\] 
where $\mathrm{R}_{\mathrm{g}_{\partial M}}$ is the Riemann curvature tensor of $\partial M$ and $\mathrm{R}_\mathrm{g}$ is that of $M$. Therefore we have
\[
\begin{aligned}
&\mathrm{Ric}_{\mathrm{g}_{\partial M}}(e_i,e_j)=\ \sum_{k=1}^{n-1}\mathrm{R}_{\mathrm{g}_{\partial M}}(e_i,e_k,e_k,e_j)\\
=\ &\sum_{k=1}^{n-1}\left(\mathrm{R}_\mathrm{g}(e_i,e_k,e_k,e_j)+\mathrm{\Rmnum{2}}(e_i,e_j)\mathrm{\Rmnum{2}}(e_k,e_k)-\mathrm{\Rmnum{2}}(e_i,e_k)\mathrm{\Rmnum{2}}(e_j,e_k)\right)\\
=\ &\mathrm{Ric}_{\mathrm{g}}(e_i,e_j)+\sum_{k=1}^{n-1}\left(\mathrm{\Rmnum{2}}(e_i,e_j)\mathrm{\Rmnum{2}}(e_k,e_k)-\mathrm{\Rmnum{2}}(e_i,e_k)\mathrm{\Rmnum{2}}(e_j,e_k)\right)-\mathrm{R}_\mathrm{g}(e_i,e_n,e_n,e_j).
\end{aligned}
\]

Hence the lower bound on $\mathrm{Ric}_\mathrm{g}$ and the two sided bound on $\mathrm{\Rmnum{2}}_{\partial M}$ are not enough to guarantee the lower bound on $\mathrm{Ric}_{\mathrm{g}_{\partial M}}$ since the information about {$``\mathrm{R}_\mathrm{g}(e_i,e_n,e_n,e_j)"$} remains unknown under these assumptions (see Example \ref{ex5.6} for {a} counterexample). This is why we need the Ricci lower bound for $\partial M$ in the setting of $\mathcal{M}(n,H,K,\lambda,D)$. {See also Remark \ref{feahiofhaeoifae}.}

	In \cite{W08}, Wong first used the assumptions of a lower bound on $\mathrm{Ric}_\mathrm{g}$ and a two sided bound on $\mathrm{\Rmnum{2}}_{\partial M}$ to derive a lower bound on $\mathrm{Ric}_{\mathrm{g}_{\partial M}}$, and consequently on $\mathrm{Ric}_{\mathrm{g}_{C_M}}$. By applying Theorem \ref{thm2.1}, he obtained a lower Ricci curvature bound for the glued space, which is a crucial step to prove Theorem \ref{thm2.3}. However, as discussed above, this line of reasoning does not hold. 
\end{remark}


\begin{defn}[Warped product of metric spaces]\label{defn2.4}
  Let $(X,\mathsf{d}_X)$, $(Y,\mathsf{d}_Y)$ be two metric spaces, and $\phi: Y\rightarrow \mathbb{R}_{+}$ be a continuous function. The warped length of a curve $\gamma= (\theta, \nu):[a,b]\rightarrow {X\times Y}$ is defined as
  \[
  L_\phi (\gamma) :=\limsup_{|\Delta|\rightarrow 0} \sum\limits_{i=1}^k \sqrt{{\phi}^2\big(\nu(s_i)\big){\mathsf{d}_X}^2(\theta(t_{i-1}),\theta(t_i))+{\mathsf{d}_Y}^2(\nu(t_{i-1}),\nu(t_i))},
  \]
  where $\Delta: a=t_0<t_1<\dots < t_k=b$, $|\Delta|=\max_{1\leqslant i\leqslant k} |t_i-t_{i-1}|$ and $s_i\in [t_{i-1},t_i]$. The warped product $\left(X\times_\phi Y,\mathsf{d}_{X\times_\phi Y}\right)$ is then defined as $X\times Y$ equipped with the metric induced by $L_\phi$.
\end{defn}

\begin{prop}[\cite{W10}]\label{prop2.5}
  Assume $\{(X_i,\mathsf{d}_i)\}$ Gromov-Hausdorff converges to $(X,\mathsf{d})$. Let $(Y,\mathsf{d}_{Y})$ be a compact metric space and $\phi: Y\rightarrow \mathbb{R}_{+}$ be a continuous function. Then $\left(X\times_\phi Y,\mathsf{d}_{X\times_\phi Y}\right)$ is isometric to the Gromov-Hausdorff limit space of $\left(X_i\times_\phi Y,\mathsf{d}_{X_i\times_\phi Y}\right)$.
\end{prop}

\section{Limit spaces of inradius collapsed manifolds}\label{sec3}

From now on, all metric spaces are considered component-wisely since {the limit spaces} may not be connected. 

Let $\{(M_i,\mathrm{g}_i)\} \subset \mathcal{M}(n,H,K,\lambda,D)$ with $\text{inrad}(M_i)\rightarrow 0$. By Theorem \ref{thm2.3}, after passing to a subsequence, we may assume that $\{(M_i,\mathsf{d}_{\mathrm{g}_i})\}$ Gromov-Hausdorff converges to some compact metric space $(Z,\mathsf{d}_Z)$.

According to the extension and smoothing out procedure in Section \ref{sec2.2}, letting $C_{M_i}:= \partial M_i \times_\phi [0,t_0]$, we obtain an RCD$\left({ -J(n,H,K,\lambda{,\varepsilon_0,t_0})},n\right)$ space $(\widetilde{M}_i,\mathsf{d}_{\widetilde{M}_i},\mathscr{H}^n)$ {with $\mathrm{diam}(\widetilde{M}_i,\mathsf{d}_{\widetilde{M}_i})\leqslant J(D{,t_0})$}, where $\widetilde{M}_i:=M_i\cup_\mathrm{id}C_{M_i}$, and the metric $\mathsf{d}_{\widetilde{M}_i}$ is defined in the same way as Remark \ref{rmk2.2}. By Theorems \ref{thm2.15} and \ref{thm2.16}, after passing to a subsequence, we may assume 
\begin{equation}\label{3.1}
\left(\widetilde{M}_i,\mathsf{d}_{\widetilde{M}_i},\frac{\mathscr{H}^n}{\mathscr{H}^n\left(\widetilde{M}_i\right)}\right)\xrightarrow{\mathrm{mGH}}(Y,\mathsf{d}_Y,\mathfrak{m}_Y),
\end{equation}
{\begin{equation}\label{3.2}
\left(C_{M_i},\mathsf{d}_{\mathrm{g}_{C_{M_i}}},\frac{\mathrm{vol}_{\mathrm{g}_{C_ {M_i}}}}{\mathrm{vol}_{\mathrm{g}_{C_ {M_i}}}(C_{M_i})}\right)\xrightarrow{\mathrm{mGH}}\left(C,\mathsf{d}_C,\mathfrak{m}_C\right),
\end{equation}}
\begin{equation}\label{3.3}
\left(\partial M_i, \mathsf{d}_{\mathrm{g}_{\partial M_i}},\frac{\mathrm{vol}_{\mathrm{g}_{\partial M_i}}}{\mathrm{vol}_{\mathrm{g}_{\partial M_i}}(\partial M_i)}\right)\xrightarrow{\mathrm{mGH}} (C_0,\mathsf{d},\mathfrak{m}),
\end{equation}
{for an RCD$({-J(n,H,K,\lambda{,\varepsilon_0,t_0})},n)$ space $(Y,\mathsf{d}_Y,\mathfrak{m}_Y)$, {an RCD$(-J(n,\lambda,K,t_0,\varepsilon_0),n)$} space {$(C,\mathsf{d}_C,\mathfrak{m}_C)$} and an RCD$(K,n)$ space $(C_0,\mathsf{d},\mathfrak{m})$.
}

%

It follows from Proposition \ref{prop2.5} that
\begin{align}\label{3.4}
\left(C,\mathsf{d}_C\right)=\left(C_0 \times_\phi [0,t_0],\mathsf{d}_{C_0 \times_\phi [0,t_0]}\right).
\end{align}
\begin{remark}\label{feahiofhaeoifae}
In convergences (\ref{3.3}) and (\ref{3.4}), the precompactness of {the sequence $\{(\partial M_i, \mathsf{d}_{\mathrm{g}_{\partial M_i}}
)\}$} in the Gromov-Hausdorff topology is important. However, {as explained in Remark \ref{rmk2.20}, we need to} add the assumption $\mathrm{Ric}_{\mathrm{g}_{\partial M_i}}\geqslant K$ ($i\in \mathbb{N}$), which is quite different from the setting in (\ref{jfieajifoa}).
\end{remark}

Let $\iota_i$ be the identity map from $ (C_{M_i},\mathsf{d}_{\mathrm{g}_{C_{M_i}}})$ to $(C_{M_i},\mathsf{d}_{\widetilde{M}_i})$, which is 1-Lipschitz. We define the limit map $\eta: (C,\mathsf{d}_C)\rightarrow (Y,\mathsf{d}_Y)$ as {the limit of $\iota_i$}, which is also 1-Lipschitz.

Now {by passing to a subsequence if necessary}, {we may assume}
\[
(M_i,\mathsf{d}_{\widetilde{M}_i})\xrightarrow{\mathrm{GH}}(X,\mathsf{d}_Y)
\]
as the convergence of subsets under (\ref{3.1}). {We set} 
\[
\eta_0:= \eta|_{C_0\times\{0\}}:C_0 \rightarrow X.
\]
Then it is clear that $\eta_0:(C_0,\mathsf{d}_C)\rightarrow (X,\mathsf{d}_Y)$ is a surjective 1-Lipschitz map.

{The following proposition is obvious. See also \cite[Lemmas 3.1, 3.2]{YZ19}.}
\begin{prop}\label{prop3.3}
The map $\eta$ satisfies the following properties.
\begin{enumerate}
\item[$(1)$] For any $(p,t)\in C\setminus C_0$, it holds that $\mathsf{d}_Y(\eta(p,t),X)=t$.
\item[{$(2)$}] $\eta: C\setminus C_0 \rightarrow Y\setminus X$ is a {bijective local isometry. Here by local isometry we mean for any $(p,t)\in C\setminus C_0$, $\eta|_{B_{t/2}(p,t)}$ is an isometry.}

\end{enumerate}
\end{prop}

{\begin{remark}[Foot point and perpendicular]\label{rmk3.3}
For any $y\in Y\setminus X$, Proposition \ref{prop3.3} guarantees the existence of a unique $(p,t)$ in $C\setminus C_0$ such that $\eta(p,t)=y$. Because $\eta$ is continuous, taking the limit $x:=\lim_{s\downarrow 0}\eta(p,s)$ yields $\eta_0(p)=x$ and
\[\mathsf{d}_Y(y,x)=\mathsf{d}_Y(y,X)=t.\]

{This} point $x$ is referred to as the {\textit{foot point}} of $y$, and the curve $\zeta:[0,t_0]\to Y$ defined by $\zeta(s)=\eta(p,s)$ is called a {\textit{perpendicular}} to $X$ at $x$.                                                                                                                                                                                              
\end{remark}}

\begin{lem}\label{lem3.4}
Let $y,z$ be two points in $Y\setminus X$ such that $\gamma_{y,z}\cap X=\emptyset$, then $\eta^{-1}(\gamma_{y,z})$ is also a minimal geodesic on $C$ connecting $\eta^{-1}(y)$ and $\eta^{-1}(z)$.
\end{lem}
\begin{proof}
For simplicity denote by $\gamma:=\gamma_{y,z}$ and by $l:=\mathsf{d}_Y(y,z)$. Then it follows from Proposition \ref{prop3.3} that $L^{\mathsf{d}_C}(\eta^{-1}(\gamma))=l$. We are done by noticing that
\[
l=\mathsf{d}_Y(y,z)\leqslant \mathsf{d}_C\left(\eta^{-1}(y),\eta^{-1}(z)\right)\leqslant L^{\mathsf{d}_C}(\eta^{-1}(\gamma))=l.
\]
\end{proof}


Let us set $\mathsf{d}_X$ as the intrinsic metric on $X$ induced by $\mathsf{d}_Y$. {We have the following estimates.}

\begin{prop}\label{prop3.5}
{There exists a constant $J=J(\lambda,\varepsilon_0,t_0)$ such that the following holds.
\begin{enumerate}
  \item[$(1)$] For any $x,y \in X$ we have  
\[
\sup\limits_{t\in [0,\mathsf{d}_Y(x,y)]} \mathsf{d}_Y\left(\gamma_{x,y}^{{ (Y,\mathsf{d}_Y)}}(t),X\right)\leqslant  J\,(\mathsf{d}_Y(x,y))^2.
\] 
\item[$(2)$]  For any $p,q\in C_0$ with $\mathsf{d}_C(p,q)\leqslant  (2J)^{-1/2}$, we have
 \[
\mathsf{d}_{C_0}(p,q)\geqslant \mathsf{d}_C(p,q)\geqslant\left(1-J\,(\mathsf{d}_C(p,q))^2\right)\mathsf{d}_{C_0}(p,q).
\]
\item[$(3)$] For any $x,y \in X$ with $\mathsf{d}_Y(x,y)\leqslant  (2J)^{-1/2}$, we have 
\[
\mathsf{d}_X(x,y)\geqslant \mathsf{d}_Y(x,y)\geqslant \left(1-J\,(\mathsf{d}_Y(x,y))^2\right)\mathsf{d}_X(x,y).
\]
\end{enumerate}}
\end{prop}
\begin{proof}{Although the proof is almost the same as \cite[Lemmas 4.2 and 4.14]{YZ19}, we give the proof for reader's convenience.}

{(1) Let $\gamma$ be a unit speed minimal geodesic from $x$ to $y$ in $Y$ and let $t^\ast\in (0,\mathsf{d}_Y(x,y))$ satisfy \[
	\mathsf{d}_Y(\gamma(t^\ast),X)=\sup_{t\in (0,\mathsf{d}_Y(x,y))}\mathsf{d}_Y(\gamma(t),X).
	\]
Take $a,b\in (0,\mathsf{d}_Y(x,y))$ such that $t^\ast\in (a,b)$, that $\gamma((a,b))\subset Y\setminus X$ and that $\gamma(a),\gamma(b)\in X$. Lemma \ref{lem3.4} guarantees the existence of a unit speed minimal geodesic $\sigma:[a,b]\to C$ such that $\gamma=\eta\circ\sigma$ on $[a,b]$.

For any sufficiently small $\epsilon>0$, take $z_i,w_i\in C_{M_i}$ such that $z_i\to \sigma(a+\epsilon)$ and $w_i\to \sigma(b-\epsilon)$ under the convergence (\ref{3.3}). Let $\gamma_i:=(\theta_i,\nu_i)$ be the unit speed minimal geodesic from $z_i$ to $w_i$ in $C_{M_i}$ and let $l_i$ be the distance between $z_i$ and $w_i$ in $C_{M_i}$. 

By Theorem \ref{thm2.17}, the limit of $\gamma_i$ must be $\sigma|_{[a+\epsilon,b-\epsilon]}$. Therefore, for sufficiently {large} $i$, we have $\gamma_i\cap \partial M_i=\emptyset$. Moreover, since for each $t\in (0,l_i)$ it holds \begin{align}\label{feaiofheaiofhaf}
\nu_i''(t)=-\mathrm{\Rmnum{2}}^{C_{M_i}^t}(\sigma_i'(t),\sigma_i'(t))=\frac{\phi'(\nu_i(t))}{\phi(\nu_i(t))}\left|\sigma_i{ '}(t)\right|^2,
\end{align}
and $\gamma_i$ is unit speed, we have $-J(\lambda,\varepsilon_0,t_0)\leqslant \sup_{[0,l_i]}\nu_i\leqslant 0$. 

Let us take $t_i\in [0,l_i]$ such that $\nu_i(t_i)=\max_{[0,l_i]}\nu_i$. It is clear that $\nu_i'(t_i)=0$. {Moreover,} there exists a real number $\xi_i$ between $t_i$ and $l_i$ such that
\begin{align}\label{feaaffaefeaf}
\nu_i({l_i})-\nu_i(t_i)=\frac{({l_i}-t_i)^2}{2}\nu_i''(\xi_i)\in [-J(\lambda,\varepsilon_0,t_0)\,{l_i}^2,0].
\end{align}

Since $\gamma_i(t_i)\to \sigma(t^\ast)$ and {$l_i\to b-a+2\epsilon$} as $i\to \infty$, letting $i\to \infty$ in (\ref{feaaffaefeaf}) and applying Proposition \ref{prop3.3} yields
\[
\mathsf{d}_Y(\gamma(t^\ast),X)\leqslant\mathsf{d}_Y(\gamma({ b-\epsilon}),X)+{J(\lambda,\varepsilon_0,t_0)}(b-a-2\epsilon)^2.
\] 
{Because $b-a\leqslant\mathsf{d}_Y(x,y)$, }letting $\epsilon\to 0$ then completes the proof.

(2) Let $\gamma:=(\theta,\nu)$ be the unit speed minimal geodesic from $p$ to $q$ in $C$. From (1), {we know} $\sup_{[0,\mathsf{d}_C(p,q)]}\nu<J\,(\mathsf{d}_C(p,q))^2$. The conclusion follows directly from Definition \ref{defn2.4} and the fact that
\[
{ 0\geqslant}\phi(\nu(t))-1=\phi(\nu(t))-\phi(0)\geqslant -J\,\nu(t),\ \forall t\in [0,\mathsf{d}_C(p,q)].
\]

(3) {The proof follows {from} the same argument as in \cite[Lemma 4.14(3)]{YZ19}, and results from (1) and (2).}
}\end{proof}

\begin{remark}\label{rmk3.6}
Since $\eta_0:(C_0,\mathsf{d}_C)\rightarrow (X,\mathsf{d}_Y)$ is a surjective 1-Lipschitz map, $\eta_0:(C_0,\mathsf{d})\rightarrow (X,\mathsf{d}_X)$ is also a surjective 1-Lipschitz map.
\end{remark}

{Now we can identify $(X,\mathsf{d}_X)$ with the Gromov-Hausdorff limit space of {$\{(M_i,\mathsf{d}_{\mathrm{g}_i})\}$} (namely $(Z,\mathsf{d}_Z)$) as follows.}

\begin{thm}\label{thm3.7}
  $(X,\mathsf{d}_X)$ is isometric to $(Z,\mathsf{d}_Z)$.
\end{thm}
\begin{proof}
{It follows directly from Proposition \ref{prop3.5} and the proof of \cite[Propostion 3.8]{YZ19}.}
\end{proof}
As for the map $\eta_0$, we have
\begin{lem}\label{lem3.8}
  For every $x \in X$, the following statement holds.
\begin{enumerate}

\item[$(1)$] $\# {\eta_0}^{-1}(x)\leqslant 2$.

\item[$(2)$] If $\#{\eta_0}^{-1}(x)=2 $, then by letting $\eta_0^{-1}(x)=\{p_1,p_2\}\subset C_0$ and taking $t\in(0,\frac{1}{2}\phi(t_0)\mathsf{d}(p_1,p_2))$, the curve {$\gamma:[-t,t]\to Y$ defined by
  \[
\begin{aligned}
   \gamma(s)=\left\{\begin{array}{ll}
    \eta\left(p_1,-s\right),&-t\leqslant s<0;\\
    \eta\left(p_2,s\right),&0\leqslant s \leqslant t,
  \end{array}
  \right.
\end{aligned}
  \]}
  is a minimal geodesic in $Y$ joining $\eta\left(p_1,t\right)$ and $\eta\left(p_2,t\right)$.
\end{enumerate}
\end{lem}

\begin{proof}
Assume there exist three distinct points $p_i\in C_0$ such that $\eta_0(p_i)=x\in X$ $(i=1,2,3)$. For $i=1,2,3$, {define the curve $\gamma_i:[0,t_0]\to Y$ as $\gamma_i(s)= \eta(p_i,s)$}, which are geodesics and perpendiculars to $X$ at $x$. We now claim that when $t<\frac{1}{2}\phi(t_0)\mathop{\mathsf{d}(p_1,p_2)}$, the curve {$\gamma_{12}:[-t,t]\to Y$ defined by
\[
\begin{aligned}
\gamma_{12}(s)=
\left\{
\begin{array}{ll}
  \gamma_1(s), &s\in(0,t);\\
  \gamma_2(-s), &s\in (-t,0),
\end{array}
\right.
\end{aligned}
\]}is a unit speed minimal geodesic joining $\gamma_1(t)$ and $\gamma_2(t)$.

For any curve $\gamma:[0,l]\rightarrow Y$ with $\gamma(0)=\eta(p_1,t)$ and $\gamma(l)=\eta(p_2,t)$, if $\gamma\cap X=\emptyset$, then due to the warped product structure of $(C,\mathsf{d}_C)$ and the local isometry property of $\eta$, we know $L(\gamma)>\phi(t_0)\mathop{\mathsf{d}(p_1,p_2)}$. If $\gamma\cap X\neq\emptyset$, then by letting $s_1:=\sup\{s|\gamma([0,s])\subset Y\setminus X\}$, $s_2:=\inf\{s|\gamma([s,l])\subset Y
\setminus X\}$, we see
\[
L\left(\gamma|_{[0,s_1]}\right)\geqslant \mathsf{d}_Y(\gamma(0),\gamma(s_1))\geqslant \mathsf{d}_Y(\gamma(0),X)= t;
\]
\[
L\left(\gamma|_{[s_2,l]}\right)\geqslant \mathsf{d}_Y(\gamma(l),\gamma(s_2))\geqslant \mathsf{d}_Y(\gamma(l),X)= t.
\]

Therefore when $t<\frac{1}{2}\phi(t_0)\mathop{\mathsf{d}(p_1,p_2)}$, $\gamma_{12}$ becomes the desired geodesic. However, we can similarly construct $\gamma_{13}$ and $\gamma_{23}$. This leads to a contradiction with Theorem \ref{thm2.17}, which is enough to conclude. 

\end{proof}

By making an identification $(X,\mathsf{d}_X)=(Z,\mathsf{d}_Z)$, we define

\begin{defn}\label{11111defn3.9}
  For $k=1,2$, define
  \[
  Z_k=X_k: =\{ x\in X | \# \eta_0^{-1}(x)=k\},
  \]
  \[
  C_0^k:= \{ p\in C_0 | \eta_0(p)\in X_k\}{.}
  \]
  
\end{defn}

\section{Metric structure of limit spaces}\label{sec4}

In this section, we still use the notation in Section \ref{sec3}. 

\subsection{Tangent space and tangent map}\label{sec4.1}

Let $p\in C_0$ and $x=\eta_0(p)$. For any monotone decreasing sequence $\{s_i\}$ with $s_i \rightarrow 0$, by Theorem \ref{thm2.16}, and by passing to a subsequence if necessary, {we may assume the existence of pointed metric spaces $\left(T_p C,\mathsf{d}_{T_p C},o_p\right)$, $\left(T_x Y,\mathsf{d}_{T_x Y},o_x\right)$ and $\left(T_p C_0,\mathsf{d}_{T_p C_0},o_p\right)$ such that
\begin{equation}\label{4.1}
\left(C,{s_i}^{-1}\mathsf{d}_C,p\right)\xrightarrow{\mathrm{pGH}} \left(T_p C,\mathsf{d}_{T_p C},o_p\right){,}
\end{equation}
\begin{equation}\label{a124.2}
\left(Y,{s_i}^{-1}\mathsf{d}_Y,x\right)\xrightarrow{\mathrm{pGH}} \left(T_x Y,\mathsf{d}_{T_x Y},o_x\right){,}
\end{equation}
\begin{align}\label{4.2}
\left(C_0,{s_i}^{-1}\mathsf{d},p\right)\xrightarrow{\mathrm{pGH}} \left(T_p C_0,\mathsf{d}_{T_p C_0},o_p\right){.}
\end{align}
After passing to a subsequence, we may also assume that {$\eta_i=\eta:\left({s_i}^{-1}C,p\right)\rightarrow \left({s_i}^{-1}Y,x\right)$ converges to a 1-Lipschitz map $\eta_\infty: \left(T_p C,o_p\right)\rightarrow \left(T_x Y,o_x\right)$.} Moreover, by taking a subsequence if necessary, we may assume that
\begin{align}\label{4.3}
\left(X,{s_i}^{-1}\mathsf{d}_Y,x\right)\xrightarrow{\mathrm{pGH}} \left(T_x X,\mathsf{d}_{T_x Y},o_x\right),
\end{align}
for some pointed metric space $\left(T_x X,\mathsf{d}_{T_x Y},o_x\right)\subset \left(T_x Y,\mathsf{d}_{T_x Y},o_x\right)$ as convergence of subsets under (\ref{a124.2}). Here we remark that the limit spaces depend on the choice of the sequence $\{s_i\}$.
}


It is worth pointing out the direct product structure of $T_p C$ as follows, which may {seem} natural since the {value of the} warped function $\phi$ is almost 1 near 0.

\begin{prop}\label{prop4.1}
$\left(T_p C,\mathsf{d}_{T_p C}\right)$ is isometric to $\left(T_p C_0 \times [0,\infty),\sqrt{\mathsf{d}_{T_p C_0}^2+\mathsf{d}_\mathbb{R}^2}\right)$, where $\mathsf{d}_\mathbb{R}$ is the standard metric on $[0,\infty)$.
\end{prop}

\begin{proof}
For any two points $q_\infty^j\in T_p C$ $(j=1,2)$ with $l:=\mathsf{d}_{T_p C}(q_\infty^1,q_\infty^2)$, assume $C\ni q_i^j \rightarrow q_\infty^j$ as $i\to \infty$ under {the} convergence (\ref{4.1}). Then the foot point $p_i^j$ of $q_i^j$ also converges to some $p_\infty^j \in T_p C_0$ with 
\[
\lim_{i\to\infty}\frac{\mathsf{d}_C\left(p_i^j,q_i^j\right)}{s_i}=\mathsf{d}_{T_p C}\left(q_\infty^j, T_p C_0\right):=r_j,
 \]
and $l_i:=\mathsf{d}_{C}(q_i^1,q_i^2)=s_i l+o(s_i)$ as $i\to \infty$.

{For each $i$, let $\gamma_i=(\theta_i,\nu_i)$ be the unit speed minimal geodesic from $q_i^1$ to $q_i^2$. Notice that for every $t\in [0,l_i]$ it holds
\[
\nu_i(t)=\mathsf{d}_C(\gamma_i(t),C_0)\leqslant \mathsf{d}_C(\gamma_i(t),\gamma_i(0))+\nu_i(0)\leqslant l_i+s_i r_1+o(s_i),\ \text{as}\ i\to \infty,
\]
which yields $\lim_{i\to\infty}\sup_{[0,l_i]}{\nu_i}=0$. 

Given any $\epsilon>0$, combining Definition \ref{defn2.4} with the fact $\lim_{t\to 0}\phi(t)={1}$ we see for all $i\gg 1$ it holds 
\[
\frac{l_i}{s_i}\geqslant \frac{1-\epsilon}{s_i}\sqrt{{\mathsf{d}_{C_0}}^2(p_i^1,p_i^2)+{(\mathsf{d}_C(q_i^1,C_0)-\mathsf{d}_C(q_i^2,C_0))}^2}
\]}{Therefore letting first $i\to \infty$ then $\epsilon\to 0$ we obtain
  \[
  \mathsf{d}_{T_p C}\left(q_\infty^1,q_\infty^2\right)\geqslant \sqrt{{\mathsf{d}_{T_p C_0}}^2 \left(p^1_\infty,p^2_\infty\right)+{(r_2-r_1)}^2}.
  \]

It now suffices to show the converse inequality. If $p_i^1=p_i^2$, then it is obvious that 
\[
\mathsf{d}_{C}(q_i^1,q_i^2)=|\mathsf{d}_{C}(q_i^2,C_0)-\mathsf{d}_{C}(q_i^1,C_0)|.
\] 
If $p_i^1\neq p_i^2$, take $\tilde{\theta_i}$ as the constant speed minimal geodesic {in $C_0$} defined on $[0,1]$ from $p_i^1$ to $p_i^2$ and define {$\tilde{\nu_i}:[0,1]\to \mathbb{R}$ as
\[ \tilde{\nu_i}(t)=\mathsf{d}_{C}(q_i^1,C_0)+(\mathsf{d}_{C}(q_i^2,C_0)-\mathsf{d}_{C}(q_i^1,C_0))\,t.
\]}Then since $\lim_{i\to\infty}\sup_{[0,1]}\tilde{\nu_i}=0$, by letting ${\tilde{\gamma}_i}:={(\tilde{\theta}_i,\tilde{\nu}_i)}$ one can similarly prove that}{
\[
\lim_{i\to \infty}\frac{\mathsf{d}_C(q_i^1,q_i^2)}{s_i}\leqslant \lim_{i\to \infty}\frac{L^C({\tilde{\gamma}_i})}{s_i}\leqslant \lim_{i\to\infty}\frac{\sqrt{{\mathsf{d}_{C_0}}^2(p_i^1,p_i^2)+{(\mathsf{d}_{C}(q_i^2,C_0)-\mathsf{d}_{C}(q_i^1,C_0))}^2}}{s_i}.
\]
This implies 
  \[
  \mathsf{d}_{T_p C}\left(q_\infty^1,q_\infty^2\right){\leqslant} \sqrt{{\mathsf{d}_{T_p C_0}}^2 \left(p^1_\infty,p^2_\infty\right)+{(r_2-r_1)}^2},
  \]
  which completes the proof.
}
\end{proof}
{\begin{remark}\label{rrmk4.2}
Assume $n\geqslant 3$ and the dimension of $(C_0,\mathsf{d})$  is $k\in [1,n-1]\cap \mathbb{N}$ (by Theorem \ref{thm2.8}). For any $p\in \mathcal{R}_k(C_0)$, it follows from \cite[Lemma 4.2]{YZ19} and the proof of Proposition \ref{prop4.1} that for all $t\in (0,t_0)$, the tangent space $(T_{(p,t)} C,o_{(p,t)})$ is unique and isometric to $(\mathbb{R}^{k+1},\mathsf{d}_{\mathbb{R}^{k+1}})$. Because $\eta$ is a local isometry, $(T_{\eta(p,t)} Y,o_{\eta(p,t)})$ is likewise unique and isometric to $(\mathbb{R}^{k+1},\mathsf{d}_{\mathbb{R}^{k+1}})$.

Moreover, by \cite[Theorem 1.3]{S20}, we have $\mathfrak{m}_Y=\eta_\sharp \mathfrak{m}_C$, where $\eta_\sharp \mathfrak{m}_C$ denotes the push forward measure of $\mathfrak{m}$ under $\eta$. Consequently, the dimension of $(Y,\mathsf{d}_Y)$ is $k+1$. 
\end{remark}}
By letting $\mathsf{d}_{T_x X}$ be the intrinsic metric on $T_x X$ induced by $\mathsf{d}_{T_x X}$, we also have the following {convexity of} $T_x X$ in $T_x Y$.
\begin{lem}\label{convex}
For any two points $x_\infty,z_\infty \in T_x X$, {and for any} unit speed minimal geodesic $\gamma_\infty:[0,l]\rightarrow T_x Y$ from $x_\infty,z_\infty$ in $T_x Y${, we have} $\gamma_\infty([0,l])\subset T_x X$. {In particular, we have $(T_x X,\mathsf{d}_{T_x X})=(T_x X,\mathsf{d}_{T_x Y})$.} 
\end{lem}
\begin{proof}
{We prove by contradiction. Assume there exists $t^\ast\in (0,l)$ such that
\[
\mathsf{d}_{T_x Y}({\gamma_\infty}(t^\ast),T_x X)=\max_{[0,l]}\mathsf{d}_{T_x Y}({\gamma_\infty}(\cdot),T_x X)>0.
\] 

Let $l'=(t^\ast+l)/2$ and $y_\infty:=\gamma_\infty(l')$.} We may assume that $\left(X,{s_i}^{-1}\mathsf{d}_Y\right)\ni x_i\rightarrow x_\infty
\in T_x X$ and $\left(Y,{s_i}^{-1}\mathsf{d}_Y\right)\ni y_i\rightarrow y_\infty \in T_x Y$ {under (\ref{4.3}) and (\ref{a124.2}) respectively}. {For each $i$, take a unit speed minimal geodesic} $\gamma_{i}:[0,l_i]\rightarrow Y$ from {$x_i$ to $y_i$ in $Y$}. By Theorem \ref{thm2.17},  {$\gamma_i$ converges to $\gamma_\infty|_{[0,l']}$} under the convergence (\ref{a124.2}). {As a result, there exists $t_i\in (0,l_i)$ such that 
\[
\frac{\mathsf{d}_{Y}(\gamma_i(t_i),X)}{s_i}=\frac{\max_{[0,l_i]}\mathsf{d}_{Y}(\gamma_i(\cdot),X)}{s_i}\to \mathsf{d}_{T_x Y}(\gamma(t^\ast),T_x X) \ \ \text{as}\ i\to \infty.
\]

Take $t_i'\in [0,t_i)$ such that $\gamma_i({t_i'})\in X$ and that $\gamma_i((t_i',t_i])\subset Y\setminus X$. {According to the proof of Proposition \ref{prop3.5}(1), we know}
\[
0=\frac{\mathsf{d}_Y(\gamma_i(t_i'),X)}{s_i}\geqslant \frac{\mathsf{d}_{Y}(\gamma_i(t_i),X)}{s_i}-\frac{J(\lambda,\varepsilon_0,t_0)\,{l_i}^2}{s_i},
\]
which leads a contradiction since $l_i/s_i\to l'$ as $i\to\infty$. }Hence we have $\gamma_\infty([0,l])\subset T_x X$. 
\end{proof}

As a direct consequence of Propositions \ref{prop3.3} and \ref{prop3.5} (see also \cite[Sublemma 4.4]{YZ19}), we have
the following proposition.
\begin{prop}\label{prop4.2}
The tangent map $\eta_\infty$ has the following properties. 
\begin{enumerate}
\item[$(1)$] For any $(p_\infty,t)\in T_p C$, it holds that $\mathsf{d}_{T_x Y}\left(\eta_\infty(p_\infty,t),T_x X\right)=t$.
\item[$(2)$] $\eta_\infty: \left(T_p C,\mathsf{d}_{T_p C} \right)\rightarrow \left(T_xY,\mathsf{d}_{T_x Y}\right)$ is 1-Lipschitz 
and {$\eta_\infty|_{T_p C \setminus T_p C_0 }:T_p C \setminus T_p C_0 \rightarrow \eta_\infty(T_p C) \setminus T_x X$ is {an}  injective local isometry in the sense of 
Proposition \ref{prop3.3}$(2)$ .}
\end{enumerate}
\end{prop}

\begin{proof}
For reader's convenience, we prove the injectivity of $\eta_\infty|_{T_p C \setminus T_p C_0 }:T_p C \setminus T_p C_0 \rightarrow \eta_\infty(T_p C) \setminus T_x X$.

Assume {there exist} two points $p_\infty^i\in T_p C_0$ $(i=1,2)$ such that {$\eta_\infty(p_\infty^1,h)=\eta_\infty(p_\infty^2,h)$ for some $h>0$. Combining the local isometry property of
 $\eta_\infty|_{T_p C \setminus T_p C_0}$ we know for any sufficiently small $\epsilon>0$ it holds 
 \[
 \mathsf{d}_{T_p C}((p^1_\infty,h),(p^2_\infty,h))=\mathsf{d}_{T_x Y}(\eta_\infty(p^1_\infty,h+\epsilon),\eta_\infty(p^2_\infty,h+\epsilon))\leqslant 2\epsilon.
 \]
 However, by Proposition \ref{prop4.1} we have $\mathsf{d}_{T_p C}((p^1_\infty,h),(p^2_\infty,h))=\mathsf{d}_{T_p C_0}(p^1_\infty,p^2_\infty)$. A contradiction.}
\end{proof}

{We also obtain an analogues version of Lemma \ref{lem3.4} as follows.
\begin{lem}\label{llem4.4}
Let $x_\infty\in \eta_\infty(T_p C\setminus T_p C_0)$. For any point $y_\infty\in T_x Y\setminus T_x X$ such that
\[
\mathsf{d}_{T_x Y}(x_\infty,y_\infty)< \mathsf{d}_{T_x Y}(x_\infty,T_x X)+\mathsf{d}_{T_x Y}(y_\infty,T_x X),
\]
we have $y_\infty\in \eta_\infty(T_p C\setminus T_p C_0)$. Moreover, {for} any geodesic connecting $x_\infty$ and $y_\infty$, there exists a unique geodesic in $T_p C\setminus T_p C_0$ whose image under $\eta_\infty$ coincides with the given geodesic.
\end{lem}
\begin{proof}
If one of the minimal geodesic from $x_\infty$ to $y_\infty$ intersects $T_x X$ at $z_\infty$, then a contradiction occurs since
\[
\mathsf{d}_{T_x Y}(x_\infty,y_\infty)= \mathsf{d}_{T_x Y}(x_\infty,z_\infty)+\mathsf{d}_{T_x Y}(y_\infty,z_\infty)\geqslant \mathsf{d}_{T_x Y}(x_\infty,T_x X)+\mathsf{d}_{T_x Y}(y_\infty,T_x X).
\]
The second statement follows directly from the 1-Lipschitz local isometry property of $\eta_\infty|_{T_p C\setminus T_p C_0}$.
\end{proof}
}

{

{The following lemma will be very useful in this paper.}
{\begin{lem}[Busemann function on $T_x Y$]\label{Busemann}Let $\tilde{v}(t):=(o_p,t), t\in [0,+\infty)$ be the perpendicular to $T_p C_0$ at $o_p$ and $v:= \eta_\infty(\tilde{v})$. For any $s>0$, set $b_s:y_\infty\mapsto \left(\mathsf{d}_{T_x Y}\left(y_\infty,v(s)\right)-s\right)$. Then the limit $b:=\lim_{s\to\infty} b_s$ exists and satisfies 
\begin{equation}\label{fffljaflkjlkdajklf}
b^{-1}(-r)=\eta_\infty(T_p C_0 \times \{r\}), \ \forall r>0.
\end{equation} 
\end{lem}
\begin{proof}

From the triangle inequality of the distance function, we know for any fixed $y_\infty \in T_x Y$, the function $s\mapsto b_s(y_\infty)$, is monotone {non-increasing} on $(0,\infty)$ with a lower bound $-\mathsf{d}_{T_x Y} (y_\infty,o_x)$. This guarantees the existence of the limit $b:=\lim_{s\to\infty} b_s$, which is {known} as the Busemann function on $T_x Y$.

Regarding (\ref{fffljaflkjlkdajklf}), we first show that if $b(y_\infty)=-h<0$, then there exists $p_\infty \in T_p C \setminus T_p C_0$ such that $\eta_\infty(p_\infty)=y_\infty$.

  For any sufficiently large $s>0$, we have $b_s(y_\infty)<-h/2<0$. Moreover, since $v$ minimizes the distance between $v(s)$ and $T_x X$, it holds that $\mathsf{d}_{T_x Y}\left(y_\infty,T_x X\right)\geqslant -b_s\left(y_\infty\right)$. From the very definition of $b_s$ we know 
 \[
 \mathsf{d}_{T_x Y}\left(y_\infty,v(s)\right)=b_s(y_\infty)+s<-b_s(y_\infty)+s\leqslant  \mathsf{d}_{T_x Y}\left(y_\infty,T_x X\right)+\mathsf{d}_{T_x Y}\left(v(s),T_x X\right).
 \] Lemma \ref{llem4.4} then yields the existence of $q_\infty \in T_p C\setminus T_p C_0$ such that $\eta_\infty (q_\infty)=y_\infty$.

 On the other hand, for any $q_\infty=(p_\infty,h) \in T_p C\setminus T_p C_0$, by letting $y_\infty=\eta_\infty(q_\infty)$ we calculate
  \begin{equation}\label{aaaaaaaaaaaaa4.4}
  \begin{aligned}
   & \mathsf{d}_{T_x Y}\left(y_\infty,v(s)\right)-s \leqslant \mathsf{d}_{T_p C}\left(q_\infty, \tilde{v}(s)\right)-s=\sqrt{\mathsf{d}_{T_p C_0}^2\left(p_\infty, o_p\right)+(s-h)^2}-s\\
    \ &=\frac{-2sh+\mathsf{d}_{T_p C}^2\left(p_\infty, o_p\right)+h^2}{\sqrt{\mathsf{d}_{T_p C_0}^2\left(p_\infty, o_p\right)+(s-h)^2}+s}\rightarrow -h,\ \ \ \ \text{as}\ s\rightarrow \infty.
  \end{aligned}
  \end{equation}This implies $b(y_\infty)<0$. Moreover, from the previous argument we see the following holds for any sufficiently large $s>0$.
    \[
  \mathsf{d}_{T_x Y}\left(y_\infty,v(s)\right)-s=\mathsf{d}_{T_p C}\left(q_\infty, \tilde{v}(s)\right)-s.
  \]
  This as well as (\ref{aaaaaaaaaaaaa4.4}) then shows $b(y_\infty)=-h$. Therefore we complete the proof of (\ref{fffljaflkjlkdajklf}).

\end{proof}}
{{An} application of Lemma \ref{Busemann} is the following result.}
\begin{prop}[Perpendiculars at points in {$X_1$} are not extendible]\label{lem4.7}
Let $p\in C_0^1$ and $x=\eta_0(p)$. Then for any $\delta\in (0,t_0)$, there does not exist a unit speed minimal geodesic {$\gamma:[-\delta,\delta]\to Y$} such that 
\[
\gamma(t)=\eta(p,t),\  \forall t\in [0,\delta].
\]
\end{prop}

{Before proceeding the proof, we may need the following lemma.}
\begin{lem}\label{prop4.6}
Let {$\gamma:[0,l]\to C$} be a unit speed minimal geodesic from $p\in C_0$ to $q\in C\setminus C_0$ such that $\gamma\cap C_0=\{p\}$, then {for some constant $J=J(\lambda,\varepsilon_0,t_0)$} we have
\begin{equation}\label{aaaaa4.13}
\frac{\mathsf{d}_C\left(\gamma(t),C_0\right)}{t}-h_{p,q}\in {\left[-Jt,0\right]}, \ \forall t\in [0,{l}],  
\end{equation}
where 
\[
h_{p,q}:=\lim_{t\downarrow 0}\frac{\mathsf{d}_C(\gamma(t),C_0)}{t}.
\]
\end{lem}
\begin{proof}
{Let us fix a sufficiently small $\epsilon>0$ and set $\bar{q}:=\gamma(l-\epsilon)$.} Assume $\partial M_i\ni p_i\rightarrow p\in C_0$ and {$C_{M_i}\ni \bar{q}_i\rightarrow \bar{q}\in C$ under (\ref{3.3}) and (\ref{3.2})} respectively.

Let $\gamma_i|_{[0,l_i]} =(\sigma_i, \nu_i)$ be the unit speed minimal geodesic from $p_i$ {to $\bar{q}_i$} in $C_{M_i}$. {Using Theorem \ref{thm2.17}, we see $\gamma_i$ converges to $\gamma|_{[0,l-\epsilon]}$ under (\ref{3.3}).}

Then by (\ref{feaiofheaiofhaf}) and the fact that $\gamma_i$ is unit speed, we have
\begin{align}\label{4.16}
-J(\lambda,\varepsilon_0,t_0)\leqslant \nu_i''(t)\leqslant 0,\ \forall t\in [0,l_i].
\end{align}
Therefore for each $i$, by using Lagrange mean-value Theorem and (\ref{4.16}), we obtain

\begin{equation}\label{sorry4.16}
\frac{\nu_i(t)}{t}-\nu_i'(0)=\frac{\nu_i(t)-\nu_i(0)}{t}-\nu_i'(0)\in \left[-{Jt},0\right],\  \forall t\in [0,l_i].
\end{equation}{This yields that for any $t_1,t_2\in [0,l_i]$ it holds that 
\begin{equation}\label{aaaaaa4.14}
\left|\frac{\nu_i(t_1)}{t_1}-\frac{\nu_i(t_2)}{t_2}\right|\leqslant \left|\frac{\nu_i(t_1)}{t_1}-\nu_i'(0)\right|+\left|\frac{\nu_i(t_2)}{t_2}-\nu_i'(0)\right|\leqslant {J}(t_1+t_2){.}
\end{equation}}{Thus letting $i\rightarrow \infty$ in (\ref{aaaaaa4.14}) implies
\[
\left|\frac{\mathsf{d}_C(\gamma(t_1),C_0)}{t_1}-\frac{\mathsf{d}_C(\gamma(t_2),C_0)}{t_2}\right|\leqslant {J}(t_1+t_2), \ \forall t_1,t_2\in \left[0,{l-\epsilon}\right], 
\]
since {$l_i\rightarrow l-\epsilon$}. Hence $h_{p,q}$ is well-defined.

To see (\ref{aaaaa4.13}), for any $0<{\delta}<t<{l-\epsilon}$, from (\ref{sorry4.16}) we have
\[
\frac{\nu_i(t)}{t}-\frac{\nu_i({\delta})}{{\delta}}=\frac{\nu_i(t)}{t}-\nu_i'(0)+\left(\frac{\nu_i({ \delta})}{{\delta}}-\nu_i'(0)\right)\in \left[-{Jt}, J{\delta}\right].
\]
Then letting $i\rightarrow \infty$ shows
\begin{equation}\label{aaaaaaaaa4.16}
\frac{\mathsf{d}_C(\gamma(t),C_0)}{t}-\frac{\mathsf{d}_C(\gamma({\delta}),C_0)}{{ \delta}}\in\left[-{Jt}, {J{ \delta}}\right]. 
\end{equation}
Hence the conclusion follows by letting ${ \delta} \rightarrow 0$ in (\ref{aaaaaaaaa4.16}) {and then letting $\epsilon\to 0$}.
}

\end{proof}
\begin{proof}[Proof of Proposition \ref{lem4.7}]
Assume the {there exists $\delta>0$ and a geodesic $\gamma:[-\delta,\delta]\to Y$} satisfying the above property. 

If there exists {$\delta'\in (0,\delta)$ (which is still denoted as $\delta$ for convenience) such that $\gamma\left([-\delta,0)\right)\subset Y\setminus X$},
then the 1-Lipschitz local isometry property of $\eta$ yields that $\eta^{-1}(\gamma|_{[-\delta,0]})$ is also a minimal geodesic on $C$. According to Lemma \ref{prop4.6}, we know
\[
h:=\lim_{s\downarrow 0}\frac{\mathsf{d}_Y\left(\gamma(-s),X\right)}{s}\geqslant \frac{\mathsf{d}_Y\left(\gamma(-\delta),X\right)}{\delta}.
\]
Using Propositions \ref{prop4.1} and \ref{prop4.2} then shows {
\[
\begin{aligned}
2=\ &\lim_{s\downarrow 0}\frac{\mathsf{d}_Y\left(\gamma(-s),\gamma(s)\right)}{s}\leqslant \liminf_{s\downarrow 0}\frac{\mathsf{d}_C\left(\eta^{-1}(\gamma(-s)),(p,s)\right)}{s}\\
\leqslant \ &\sqrt{(1-h)^2+1-h^2}< \sqrt{2},\\
\end{aligned}
\]which is impossible.}

{According to the continuity of the map $t\mapsto \mathsf{d}_Y(\gamma(t),X)$, it suffices to consider the case that there exists a sequence $t_i\rightarrow 0$ such that $\gamma(-t_i)\in X$ for all $i$. 

{Plugging $s_i={t_i}^{3/2}$ into (\ref{4.1})-(\ref{4.3}) and passing to a subsequence if necessary, we obtain tangent spaces and tangent maps denoted with the same notation for convenience.} In addition, we may also require that $\gamma_i:=\gamma|_{[-t_i,t_i]}$ converges to a limit geodesic $\gamma_\infty$ in $T_x Y$ {under (\ref{a124.2})}. 

Since $t_i/s_i\rightarrow \infty$, we see $\gamma_\infty$ is a line. {From Proposition \ref{prop3.5} we know 
\[
\lim_{i\to\infty}\sup_{t\in [-t_i,0]}\frac{\mathsf{d}_Y\left(\gamma(t),X\right)}{s_i}\leqslant {J(\lambda,\varepsilon_0,t_0)}\lim_{i\to \infty}\frac{t_i^2}{s_i}=0,
\]
which yields $\gamma_\infty((-\infty,0])\subset T_x X$.

Notice that $\gamma_\infty(t)=\eta_\infty(o_p,t)$ for any $t>0$. Let us {consider} the Busemann function {$b$ on $T_x Y$ as} in Lemma \ref{Busemann}. Then from our setting we know 
\[
b(\gamma_\infty(-t))=t,\ \forall t>0.
\] 

For each {$t>0$}, let $\zeta_t:[0,1]\rightarrow T_x Y$ be the perpendicular to $T_x X$ at $\gamma_\infty(-t)$. We claim
\begin{equation}\label{eqn5122222222}
2h\leqslant \mathsf{d}_{T_x Y}\left(\zeta_t(h),\gamma_\infty(h)\right)\leqslant 2h+t,\ \forall t,h\in (0,1].
\end{equation}

If the first inequality does not hold for some $t$ and $h$, then Lemma \ref{llem4.4} would imply the existence of $q_\infty\in T_p C\setminus T_p C_0$ such that $\eta_\infty(q_\infty)={\zeta_t}(h)$. However, by Lemma \ref{Busemann} we see $b(\gamma_\infty(-t))=0$, contradicting our assumption that $b(\gamma_\infty(-t))=t$.

Now as $t\downarrow 0$, $\zeta_t$ converges to a unit speed minimal geodesic $\zeta$ from $o_x$ to $w_\infty:=\lim_{t\downarrow 0}\zeta_t(1)$. Moreover, from (\ref{eqn5122222222}) we know that the curve
\[
\tilde{\zeta}{(t)}:=
\left\{\begin{aligned}
\zeta({-t}),\ \ &\ \ \text{if }t\in [-1,0],\\
\gamma_\infty(t), \ \ &\ \ \text{if }t\in [0,1],
\end{aligned}
\right.
\] 
is a minimal geodesic from $w_\infty$ to $\gamma_\infty(1)$. This contradicts Theorem \ref{thm2.17}.
}}
\end{proof}
}

{Another application of Lemma \ref{Busemann} is the following proposition. 
\begin{prop}\label{surjmapeta}
The map $\eta_\infty|_{T_p C_0}: T_p C_0\rightarrow T_x X$ is surjective.
\end{prop}
\begin{proof}
{{Fix an arbitrary point $x_\infty\in T_xX$. For} any sufficiently large $s>0$, {define the geodesic $\gamma_s:[0,1]\longrightarrow T_x Y$ as
  \[
  \gamma_s(t)=\gamma^{T_x Y}_{\eta_\infty(o_p,s),x_\infty}(s-t).
  \]}
  
  Since the triangle inequality implies $\mathsf{d}_{T_x Y}\left(\gamma_s(1),x_\infty\right)\leqslant \mathsf{d}_{T_x X}(o_x,x_\infty)+1$, $\gamma_s$ uniformly converges to a geodesic $\gamma_\infty:[0,1]\rightarrow T_x Y$ such that
\begin{equation}\label{equatione4.17}
\gamma_\infty(t)\in {{b}^{-1}}(-t),\ \forall t\in (0,1].
\end{equation}
{Here $b$ is the Busemann function {on $T_xY$} defined in Lemma \ref{Busemann}.}

Therefore, (\ref{fffljaflkjlkdajklf}) and the local isometry property of $\eta_\infty|_{T_p C\setminus T_p C_0}$ verifies the existence of a point $p_\infty\in T_p C_0$ such that \[
\gamma_\infty(t)=\eta_\infty(p_\infty,t),\ \forall t\in [0,1].
\]

If $\eta_\infty(p_\infty)\neq x_\infty$, then by our construction it holds that ${b}(x_\infty)>0$. However, this would lead to a contradiction, as can be shown by an argument analogous to that in the proof of Proposition \ref{lem4.7}. 

Hence we have \begin{align}\label{fhoeahofe}
	\gamma_\infty(0)=\eta_\infty(p_\infty)=x_\infty,
	\end{align}}
from which we conclude.

\end{proof}
\begin{cor}\label{cor4.10}
{The $0$-level set of the Busemann function satisfies $b^{-1}(0)=T_x X$.}
\end{cor}
{\begin{proof}
	{For any $x_\infty\in T_x X$, by using notations in the proof of  Proposition \ref{surjmapeta}, it follows from  (\ref{equatione4.17}), (\ref{fhoeahofe}) and the continuity of Busemann function that $b(x_\infty)=0$, which implies} $T_x X\subset b^{-1}(0)$.

	To prove the reverse inclusion, let $x_\infty\in b^{-1}(0)$. For any sufficiently large $s>0$, define the point \[
	y^s_\infty:=\gamma^{T_x Y}_{\eta_\infty(o_p,s),x_\infty}(s-1).
	\]
	By {the definition of the function $b_s$ in Lemma \ref{Busemann}}, this point satisfies $b_s(y^s_\infty)=-1$ and 
	\[
	b_s(x_\infty)=\mathsf{d}_{T_x Y}(x_\infty,y_\infty^s)-1.
	\]
	Because the function $s\mapsto b_s(x_\infty)$ is monotone non-increasing,  $y^s_\infty$ converges to a limit point $y_\infty\in b^{-1}({-1})$ as $s\to\infty$. Moreover, we have $\mathsf{d}_{T_x Y}(x_\infty,y_\infty)=1$. 
	
The conclusion follows by combining these results with Proposition \ref{prop4.2} and Lemma \ref{Busemann}. 
\end{proof}
}
{
\begin{cor}\label{cor4.11}
	For any $x_\infty\in T_x X$ we have  $1\leqslant \# {\eta_\infty}^{-1}(x_\infty)\leqslant 2$.
\end{cor}
\begin{proof}
By Proposition \ref{surjmapeta}, it suffices to show $\# {\eta_\infty}^{-1}(x_\infty)\leqslant 2$. Assume there exists three distinct points $p_\infty^i$ ($i=1,2,3$) such that $\eta_\infty(p_\infty^i)=x_\infty$.

Let $t_\infty=\min_{i\neq j}(\mathsf{d}_{T_p C_0}(p_\infty^i,p_\infty^j))/2$ and $v_{ij}:[-t_\infty,t_\infty]\to T_x Y$ be the curve defined by
\[
v_{ij}(t)=\left\{\begin{aligned}
 \eta_\infty(p_\infty^i,t),\ \ \ &t\in (0,t_\infty];\\
  \eta_\infty(p_\infty^j,-t),\ &t\in [-t_\infty,0].\\
\end{aligned}\right.
\]

{Since $t_\infty<\mathsf{d}_{T_p C_0}(p_\infty^i,p_\infty^{j})$,} 
	according to the proof of Lemma \ref{lem3.8}, we know $v_{ij}$ is exactly the unit speed minimal geodesic connecting $\eta_\infty(p_\infty^i,t_\infty)$ and $\eta_\infty(p_\infty^j,t_\infty)$. This contradicts Theorem \ref{thm2.17}.
\end{proof}

}
{When $x\in X_2$, since the tangent space $T_x X$ contains a line, we obtain the following Proposition.}
\begin{prop}\label{lem4.4}
	For any $x\in X_2$, the tangent space $(T_x Y,\mathsf{d}_{T_x Y})$ is isometric to $(T_x X\times \mathbb{R},\sqrt{{\mathsf{d}_{T_x{X}}}^2+{\mathsf{d}_{\mathbb{R}}}^2})$ and $\eta_\infty:(T_p C_0,\mathsf{d}_{T_p C_0})\rightarrow (T_x X,\mathsf{d}_{T_x X})$ is an isometry.
\end{prop}
\begin{proof}
	According to Theorem \ref{CCsplit},   there exists a metric space $\left(\Omega,\mathsf{d}_\Omega\right)$ and a point $\omega\in \Omega$ such that $(T_x Y,\mathsf{d}_{T_x Y})$ is isometric to $\left(\Omega\times \mathbb{R},\sqrt{\mathsf{d}_{\Omega}^2+\mathsf{d}_{\mathbb{R}}^2}\right)$ and that 
	\[
	\eta_\infty(o_p,t)=(\omega,t),\ \forall t>0.
	\]

{Let us take the Busemann function {$b$ on $T_xY$} as in Lemma \ref{Busemann}. From the construction of Busemann function and Corollary \ref{cor4.10}, we know 
$b^{-1}(0)=\Omega\times\{0\}=T_x X$ and
\[
b^{-1}(-r)=\eta_\infty(T_p C_0\times\{r\})=\Omega\times\{r\},\ \forall r>0.
\]

For every $r>0$, since $\eta_\infty|_{B_{r}^{T_p C_0}(o_p)\times\{4r\}}$ is an isometry, $\eta_\infty:B_{r}^{T_p C_0}(o_p)\to B_r^\Omega(\omega)$ is also an isometry. By identifying $\Omega$ with $T_x X$, it follows that $\eta_\infty:T_p C_0\to T_x X$ is an isometry.
}
 
\end{proof}

{In the remainder of this subsection, we focus on proving the following proposition. By Proposition \ref{lem4.4}, it suffices to consider the case where $p\in C_0^1$. From this point onward, we will assume $p\in C_0^1$ without further explicit mention.
}

\begin{prop}\label{prop4.3}
	{For any $p_\infty\in T_p C_0$ it holds that}
	\begin{align}\label{4.7}
		\mathsf{d}_{T_p C_0}(o_p,p_\infty)=\mathsf{d}_{T_x X}(o_x,\eta_\infty(p_\infty)).
	\end{align}
\end{prop}
{We start with the following lemma.}
{\begin{lem}\label{feahiohfeaio}
For any $t>0$ and $p_\infty\in T_p C_0$, any  minimal geodesic  starting from $\eta_\infty(p_\infty,t)$ intersects $T_x X$ at most at one point.
\end{lem}
\begin{proof}
Let us prove by contradiction. By Lemma \ref{convex}, without loss of generality, we may assume there exists a unit speed minimal geodesic $\gamma:[-1,1]\to T_x Y$ such that $\gamma(-1)=\eta_\infty(p_\infty,h)$ for some $h\in (0,1)$, that $\gamma([-1,0))\subset T_x Y\setminus T_x X$ and that $\gamma([0,1])\subset T_x X$. The local isometry property and 1-Lipschitz property of $\eta_\infty$ {imply} that ${\eta_\infty}^{-1}\left(\gamma|_{[-1,0)}\right)$ is also a minimal geodesic in $T_p C$. Indeed, ${\eta_\infty}^{-1}\left(\gamma|_{[-1,0]}\right)$ is a minimal geodesic from $(p_\infty,h)$ to $q_\infty:=\lim_{s{\uparrow 0}}{\eta_\infty}^{-1}(\gamma(s))\in T_p C_0$.

Let $x_\infty=\eta_\infty(q_\infty)$. Denote by $\zeta:t\mapsto \eta_\infty(q_\infty,t)$ for convenience. Let us fix a sequence $r_i\downarrow 0$ and $a>0$ such that
\[
\sqrt{1-h^2}+\sqrt{h^2+a^2}<1+a.
\]For each $i\in\mathbb{N}$, we take {a perpendicular $\zeta_i$ to $T_x X$} at $\gamma(r_i)$.

\textbf{Case 1} There exists a subsequence which is still denoted as $r_i$ such that
\[
\lim_{i\to \infty}\mathsf{d}_{T_x Y}(\zeta_i(1),\zeta(1))=2
\]
Then {the curve} $v:[-1,1]\to T_x Y$ defined as 
\[
v(t)=\left\{\begin{aligned}
\zeta(t),\ \ \ \ \ \ \ \ \ &t\in (0,1];\\
\lim_{i\to\infty}\zeta_i(-t),\ &t\in [-1,0],
\end{aligned}\right.
\]
is a unit speed minimal geodesic such that
\begin{align}\label{fffeaniofeahoi}
\mathsf{d}_{T_x Y}\left(v(t),T_x X\right)=|t|, \ \forall t\in [-1,1].
\end{align}

By Theorem \ref{CCsplit}, after passing to a subsequence, we may assume $({r_i}^{-1}T_xY,x_\infty)$ pGH converges to a pointed product space $(\Omega\times \mathbb{R},(\omega,0))$ and $v$ converges to the line $v_\infty:t\mapsto (\omega,t)$ {under this convergence}.

Since (\ref{fffeaniofeahoi}) implies 
\[
\mathsf{d}_{T_x Y}\left(v(t),\gamma(s)\right)\geqslant |t|, \ \forall t\in [-1,1], \forall s\in [0,1],
\]
we know $\gamma|_{[0,r_i]}$ converges to unit speed minimal geodesic $\gamma_\infty:[0,1]\to \Omega$ as $i\to \infty$. Therefore, we obtain\[
\lim_{i\to\infty}\frac{\mathsf{d}_{T_x Y}\left(\zeta(h r_i ),\gamma(ar_i)\right)}{r_i}=\sqrt{h^2+a^2}.
\]

According to the 1-Lipschitz property of $\eta_\infty$, we see $\mathsf{d}_{T_x Y}\left(\gamma(-r_i),\zeta(hr_i)\right)\leqslant r_i\sqrt{1-h^2}$. A contradiction then occurs because we have
\begin{equation}\label{faehiofhaeiohfoae}
\begin{aligned}
&1+a=\lim_{i\to\infty}\frac{\mathsf{d}_{T_x Y}(\gamma(-r_i),\gamma(ar_i))}{r_i}\\
\leqslant\ & \limsup_{i\to\infty}\frac{\mathsf{d}_{T_x Y}\left(\gamma(-r_i),\zeta(h r_i )\right)}{r_i}+\limsup_{i\to\infty}\frac{\mathsf{d}_{T_x Y}\left(\zeta(h r_i ),\gamma(ar_i)\right)}{r_i}\\
\leqslant\ &  \sqrt{1-h^2}+\sqrt{h^2+a^2}<1+a. 
\end{aligned}
\end{equation}

\textbf{Case 2}  $\limsup_{i\to \infty}\mathsf{d}_{T_x Y}(\zeta_i(1),\zeta(1))<2$.

By Lemma \ref{llem4.4}, there exists {$\{q_i\}\subset T_p C_0$ such that for each $i$ we have $\eta_\infty(q_i)=\gamma(r_i)$ and $\mathsf{d}_{T_p C_0}(q_i,q_\infty)<2$}. If $q_i$ converges to some $q_\infty'\neq q_\infty$, then following the argument in Lemma \ref{lem3.8}, the geodesic $v'$ defined as
\[
v'(t)=\left\{\begin{aligned}
\zeta(t),\ \ \ \ \ \ \ \ \ \ \ \ &\ 0\leqslant t\leqslant \mathsf{d}_{T_p C_0}(q_\infty,q_\infty');\\
\eta_\infty(q_\infty',-t),\ \ &\  -\mathsf{d}_{T_p C_0}(q_\infty,q_\infty')\leqslant t<0,
\end{aligned}\right.
\]
would satisfy
\[
\mathsf{d}_{T_x Y}\left(v'(t),T_x X\right)=|t|, \ \forall t\in  [-\mathsf{d}_{T_p C_0}(q_\infty,q_\infty'),\mathsf{d}_{T_p C_0}(q_\infty,q_\infty')].
\]
However, this is impossible due to Case 1. Therefore, one must have $\lim_{i\to\infty }q_i=q_\infty$.

Let $t_i=\mathsf{d}_{T_p C_0}(q_\infty,q_i)$. Since $\eta_\infty$ is 1-Lipschitz, it follows that $t_i\geqslant r_i$. Applying Theorem \ref{CCsplit} and Lemma \ref{convex}, after passing to a subsequence, we may assume 
\begin{align}\label{feahofhae}
({t_i}^{-1}T_xY,x_\infty)\xrightarrow{\mathrm{pGH}}(\Omega\times \mathbb{R},(\omega,0))
\end{align} for a pointed product space $(\Omega\times \mathbb{R},(\omega,0))$, {and}
\[
({t_i}^{-1}T_x X,x_\infty)\xrightarrow{\mathrm{pGH}}(T_{x_\infty}(T_x X),(\omega,0))
\] for a pointed metric space $(T_{x_\infty}(T_x X),(\omega,0))$ as convergence of subsets. Under (\ref{feahofhae}), the geodesic $\gamma$ converges to the line $\gamma_\infty:t\mapsto (\omega,t)$ with $\gamma_\infty(0)=(\omega,0)$ and $\gamma_\infty([0,\infty))\subset T_{x_\infty}(T_x X)$, and $\zeta$ {converges} to $\zeta_\infty$, which is the perpendicular to $T_{x_\infty}(T_x X)$ at $(\omega,0)$.

If $\lim_{i\to \infty}r_i/t_i=0$, then following the same line of reasoning as before, we deduce the existence of $h_\infty>0$ and a geodesic $v_\infty:[-h_\infty,h_\infty]\to T_{x_\infty}(T_x Y)$ such that $v_\infty$ coincides with $\zeta_\infty$ on $[0,h_\infty]$ and that
\[
\mathsf{d}_{T_{x_\infty}(T_x Y)}(v_\infty(t),T_{x_\infty}(T_x X))=|t|,\ \forall t\in [-h_\infty,h_\infty]. 
\]
This would also imply 
\[
\lim_{t\downarrow 0}\frac{\mathsf{d}_{T_{x_\infty}(T_x Y)}(\zeta_\infty(ht),\gamma_\infty(at))}{t}=\sqrt{h^2+a^2}.
\]

Additionally, since
\[
\mathsf{d}_{T_{x_\infty}(T_x Y)}\left(\gamma_\infty(-t),\zeta_\infty(ht)\right)\leqslant t\sqrt{1-h^2},\ \forall t>0,
\]
we arrive at a contradiction {in a way similar} to (\ref{faehiofhaeiohfoae}).

From the above discussion, we know there exists a map $\tau:[0,1]\to T_p C_0$ such that $\gamma=\eta_\infty\circ \tau$ on $[0,1]$ and that 
\begin{align}\label{faehiofiaeohfia}
\liminf_{t\downarrow 0}\frac{t}{\mathsf{d}_{T_p C_0}(q_\infty,\tau(t))}>0.
\end{align} 

For each $t>0$, let $\zeta_{i,t}:[0,\infty)\to T_x Y$ denote the ray $\zeta_{i,t}(s)=\eta_\infty(\tau(r_i t),s)$. Assume $\zeta_{i,t}$ converges to $\zeta_{\infty,t}$ under  (\ref{feahofhae}), which is the perpendicular to $T_{x_\infty}(T_x X)$ at $\gamma_\infty(t)$. Then $\zeta_{\infty,t}$ must be a ray in $\Omega\times\{t\}$.

However, by (\ref{faehiofiaeohfia}) and the 1-Lipschitz property of $\eta_\infty$, $\zeta_{\infty,t}$ converges to $\zeta_\infty$ as $t\downarrow 0$, yielding $\zeta_\infty\in \Omega\times\{0\}$. This leads to another contradiction because
\[
\mathsf{d}_{T_{x_\infty}(T_x Y)}(\zeta_\infty(h),\gamma_\infty(a))=\sqrt{h^2+a^2}.
\]
\end{proof}}

{\begin{lem}\label{lemma4.14}
	Let $p\in C_0^1$ and assume $p_\infty\in T_p C_0$ satisfies ${\eta_\infty}^{-1}(\eta_\infty(p_\infty))={\{p_\infty\}}$. Then $\eta_\infty|_{T_p C_0}$ preserves the distance from $p_\infty$.
\end{lem}
\begin{proof}
By Lemma \ref{feahiohfeaio}, for any $q_\infty\in T_p C_0$ and any {$s>0$, any minimal geodesic $\gamma_{s}$ from $\eta_\infty(p_\infty)$ to $\eta_\infty(q_\infty,s)$ satisfies $\gamma_{s}\cap T_x X=\{\eta_\infty(p_\infty)\}$.} Since $\eta_\infty|_{T_p C\setminus T_p C_0}$ is a 1-Lipschitz local isometry, {${\eta_\infty}^{-1}(\gamma_{s})$} is also a minimal geodesic on $T_p C$ connecting {$p_\infty$ and $(q_\infty,s)$}. This gives the equality of distances
\[
{\mathsf{d}_{T_p C}(p_\infty,(q_\infty,s))=\mathsf{d}_{T_x Y}(\eta_\infty(p_\infty),\eta_\infty(q_\infty,s)).}
\]
The conclusion follows by taking the limit as $s\to 0$.
\end{proof}}

{We are now in a position to give a proof of Proposition \ref{prop4.3}.}

\begin{proof}[Proof of Proposition \ref{prop4.3}]
Assume $p_\infty\in T_p C_0\setminus \{o_p\}$.  Let $\sigma:[0,l]\to T_x X$ be a unit speed minimal geodesic from $o_x$ to $\eta_\infty(p_\infty)$. 


For any sufficiently small $\epsilon>0$, define
\[
A_\epsilon:=\{t\in[0,l]|\,\mathsf{d}_{T_p C_0}({\eta_\infty}^{-1}(\sigma(s)),o_p)\leqslant (1+\epsilon)s,\ \forall s\in [0,t]\},
\]
{which} is closed due to the continuity of $\eta_\infty$ {and is non-empty because $0\in A_\epsilon$}. 

To show $A_\epsilon$ is open, {choose any $t\in A_\epsilon$. If $t$} satisfies $\# {\eta_\infty}^{-1}(\sigma(t))=1$, then Lemma \ref{lemma4.14} yields 
\begin{equation}\label{fehuafheoaihf}
	\begin{aligned}
		&\mathsf{d}_{T_p C_0}({\eta_\infty}^{-1}(\sigma(s)),o_p)\leqslant (1+\epsilon)t+\mathsf{d}_{T_p C_0}({\eta_\infty}^{-1}(\sigma(t)),{\eta_\infty}^{-1}(\sigma(s)))\\
		=\ &(1+\epsilon)t+s-t\leqslant(1+\epsilon)s,\ \forall s\in (t,l].
	\end{aligned}
\end{equation}

{{By Corollary \ref{cor4.11}}, it suffices to consider the case that {$\# {\eta_\infty}^{-1}(\sigma(t))=2$. Take} two points $p_t^1,p_t^2\in T_p C_0$ such that $\mathsf{d}_{T_p C_0}(p_t^1,p_t^2)=2l_t$ and $\eta_\infty(\{p_t^1,p_t^2\})=\sigma(t)$. Define the geodesic $v:[-l_t,l_t]\to T_x Y$ by
\[
v(s)=\left\{
\begin{aligned}
	\eta_\infty(p_t^1,s),\ \ \ \ \ &{s}\in [0,l_t];\\
	\eta_\infty(p_t^2,-s),\ \ \ &s\in [-l_t,0),
\end{aligned}\right.
\]
which satisfies\begin{align}\label{4riohaoifho}\mathsf{d}_{T_x Y}(v(s),T_x X)=|s|,\ \forall s\in [-l_t,l_t].\end{align}}{Moreover, from the proof of Corollary \ref{cor4.11} we know $v$ is a unit speed minimal geodesic from $\eta_\infty(p_t^2,l_t)$ to $\eta_\infty(p_t^1,l_t)$. }

{We claim there exists some small $\mu>0$ such that for any $s\in [t,t+\mu]$ there exists $p_s^i\in {\eta_\infty}^{-1}(\sigma(s))$ such that
	\begin{align}\label{fehifaohefioahf}
	\mathsf{d}_{T_p C_0}(p_t^i,p_s^i)\leqslant (s-t){ (1+\epsilon)},\ i=1,2.
	\end{align}

	{Assume there exists a sequence} $r_j\downarrow 0$ such that 
\begin{align}\label{feafefa}
\mathsf{d}_{T_p C_0}(p_t^1,{\eta_0}^{-1}(\sigma(t+r_j)))>(1+\epsilon)r_j,\ \forall j\in\mathbb{N}.
\end{align}

By Theorem \ref{CCsplit} and (\ref{4riohaoifho}), {after passing to a subsequence,} there exists a metric space $(\Omega,\mathsf{d}_\Omega)$ such that as $j\to\infty$, $({r_j}^{-1}T_x Y,\sigma(t))$ pGH converges to a product metric space $(\Omega\times\mathbb{R},(\omega,0))$. Moreover, $\sigma(t+r_j)$ converges to a point $(\omega',0)\in \Omega\times\{0\}$ with $\mathsf{d}_\Omega(\omega,\omega')=1$.

{For every $\tau>0$ and $j\in\mathbb{N}$,} since 
\[
\eta_\infty:B_{{\tau}r_j}^{T_p C}\left((p_t^1,5{\tau }r_j)\right)\to \eta_\infty(B_{{\tau}r_j}^{T_p C}\left((p_t^1,5{\tau }r_j)\right))
\]
is an isometry, $\eta_\infty(B_{{\tau}r_j}^{T_p C_0}(p_t^1)\times\{5{\tau }r_j\})$ converges to $B_{\tau}^\Omega(\omega)\times\{5{\tau}\}$ under the above pGH convergence. {In particular, we know $B_{\tau r_j}^{T_p C_0}(p_t^1)$ GH converges to $B_{\tau}^\Omega(\omega)$.} As a result, {we have
\[
\eta_\infty(B_{2r_j}^{T_p C_0}(p_t^1))\subset B_{3r_j}^{T_x Y}(\sigma(t))\subset \eta_\infty(B_{2r_j}^{T_p C_0}(p_t^1))
\]
for all sufficiently large $j$. Therefore,} there exists $p^1_{r_j}\in {\eta_\infty}^{-1}(\sigma(t+r_j))$ such that 
\[
\lim_{j\to\infty}\frac{\mathsf{d}_{T_p C_0}(p^1_{r_j},p^1_t)}{r_j}=1,
\]
contradicting (\ref{feafefa}). 

Thus we have (\ref{fehifaohefioahf}).} Because this implies {$A_\epsilon=[0,l]$,} the conclusion follows by letting $\epsilon\to 0$. 
\end{proof}

\subsection{Geometric properties of limit spaces}\label{sec4.2}

{This subsection is aimed at proving Theorem \ref{thm1.3}.} {By Lemma \ref{lem3.8}, we can define $f:C_0\rightarrow C_0$ as}
\begin{equation}\label{feaohfoiaehfiohae}
f(p)=
\left\{
\begin{aligned}
	q\ \ & \text{if $p\in C_0^2$ and $\{p,q\}={\eta_0}^{-1}(\eta_0(p))$},\\
	p\ \ & \text{if $p \in C_0^1$}.
\end{aligned}
\right.
\end{equation}

We first deal with the following theorem. 

\begin{thm}\label{thmfffff4.5}
$f:C_0\rightarrow C_0$ is an isometry.
\end{thm}

{To prove this theorem, the continuity of $f$ is needed.}

\begin{prop}\label{prop4.8}
 $f:C_0\rightarrow C_0$ is a homeomorphism.
\end{prop}

\begin{proof}
 Since $f$ is an involution, it suffices to show that $f$ is continuous. {Let $p_i\rightarrow p$ in $C_0$, and
 $x_i:=\eta_0(p_i)$, $x:=\eta_0(p)$}.
  We consider the following two cases.

\textbf{Case 1} \,{ $p\in C_0^1$.}

Since each component of $C_0$ is compact by Theorem \ref{thm2.3}, it suffices to show that any convergent subsequence of $\{f(p_i)\}$ (which is still denoted as $\{f(p_i)\}$ for simplicity) must converge to $p$. 

Assume $f(p_i)\rightarrow q$ for some $q\in C_0$. Then the continuity of $\eta_0$ implies that 
\[
\eta_0(q)=\lim_{i\rightarrow \infty}\eta_0(f(p_i))=\lim_{i\rightarrow \infty}x_i=x.
\]
Thus $q=p$ and we are done.

\textbf{Case 2}\, { $p\in C_0^2$.}

{Let $s_i:=\mathsf{d}_X(x_i,x)$. Passing to a subsequence if necessary, we may assume that $({s_i}^{-1}C_0,p)\to (T_p { C_0},o_p)$, $({s_i}^{-1}C_0,f(p))\to \left(T_{f(p)}{ C_0}, o_{f(p)}\right)$ and $({s_i}^{-1}X,x)\to (T_x X,o_x)$ {as $i\to\infty$.} By {Propositions \ref{prop4.1} and \ref{lem4.4}}, $(T_p { C_0},o_p)$ is isometric to $\left(T_{f(p)}{ C_0}, o_{f(p)}\right)$. {Passing to a subsequence again,} $\eta:({s_i}^{-1}C_0,f(p))\to ({s_i}^{-1}X,x)$ also converges to an isometry {$\eta_\infty:\left(T_{f(p)}C_0,o_{f(p)}\right)\to (T_x X,o_x)$}. In particular, for any sufficiently large $i$ it holds
\[
B^X_{2s_i}(x)\subset \eta_0\left(B^{C_0}_{3s_i}(f(p))\right)\subset B^X_{4s_i}(x),
\]
from which we know $x_i\in X_2$ and $f(p_i)\to f(p)$.

}

\end{proof}

\begin{cor}\label{cor4.15}
$C_0^1$ and $X_1$ are closed in $C_0$ and $X$ respectively.
\end{cor}

\begin{proof}[Proof of Theorem \ref{thmfffff4.5}]


{Fix $p,q\in C_0$ such that $\mathsf{d}(p,q)=l<\infty$. Let $\gamma$ be a unit speed minimal geodesic from $p$ to $q$.} For any $\delta>0$, set 
\[
A_\delta:=\left\{t\in [0,l]|\,{s}\geqslant (1-\delta)\mathsf{d}\big(f(p),f(\gamma(s))\big), \ \forall 0\leqslant s\leqslant t\right\}.
\]


Obviously $A_\delta$ is a closed set due to Proposition \ref{prop4.8}. If $\sup A_\delta=l_0<l$, {we claim there exists $\mu\in (0,l-l_0)$ such that
\begin{equation}\label{sorry4.20}
{s-l_0}=\mathsf{d}\left(\gamma(l_0),\gamma(s)\right)\geqslant (1-\delta) \mathsf{d}\left(f(\gamma(l_0)),f(\gamma(s))\right),\ \forall {s\in (l_0,l_0+\mu)}.
\end{equation}

Assume {there exists} $s_i\downarrow l_0$ such that (\ref{sorry4.20}) does not hold for every $s_i$. Then applying Propositions \ref{prop4.3} and \ref{prop4.8} yields that 
\[
1=\limsup_{i\rightarrow \infty} \frac{\mathsf{d}(\gamma (l_0),\gamma(s_i))}{\mathsf{d}_X(\gamma (l_0),\gamma(s_i))}\leqslant (1-\delta)\limsup_{i\rightarrow \infty} \frac{\mathsf{d}\left(f(\gamma (l_0)),f(\gamma(s_i))\right)}{\mathsf{d}_X\left(\gamma (l_0),\gamma(s_i)\right)}=1-\delta,
\]
which is a contradiction. 
}

Therefore for any $s\in [l_0,l_0+\mu)$ it holds{
\[
\begin{aligned}
s\geqslant(1-\delta) \Big(\mathsf{d}\big(f(p),f(\gamma(l_0))\big)+\mathsf{d}\big(f(\gamma(l_0)),f(\gamma(s))\big)\Big)\geqslant (1-\delta)\mathsf{d}\big(f(p),f(\gamma(s))\big), 
\end{aligned}
\]}contradicting $l_0<l$. Thus $\sup A_\delta=\max A_\delta=l$. Letting $\delta\rightarrow 0$ then implies {$\mathsf{d}(p,q)\geqslant \mathsf{d}(f(p),f(q))$.
} 

The converse inequality $\mathsf{d}(f(p),f(q))\geqslant \mathsf{d}(p,q)$ follows from the fact that $f$ is an involution. This completes the proof.
\end{proof}




{The following theorem can also be proved using the same method as in Theorem \ref{thmfffff4.5}. For brevity, we only outline the key steps of the proof.}

\begin{thm}\label{thm4.10}
  $\left(X,\mathsf{d}_X\right)$ is isometric to the quotient space $\left(C_0/  f,\mathsf{d}^\ast\right)$, where $\mathsf{d}^\ast$ is the quotient metric defined as in $($\ref{a2.3}$)$. {In particular, $(X,\mathsf{d}_X)$ is geodesically non-branching.}
\end{thm}

\begin{proof}

{Fix $x,y\in X$ with $\mathsf{d}(x,y)=l$.} Given any $\delta>0$ we set
\[
A_\delta:=\left\{t\in [0,l]|\,s\geqslant (1-\delta)\,{ \mathsf{d}} (\eta^{-1}(x),\eta^{-1}(\gamma(s))), \ \forall 0\leqslant s\leqslant t\right\}.
\]

{By Proposition \ref{prop4.3} and Theorem \ref{thmfffff4.5} and the continuity argument,	we conclude that $\sup A_\delta=l$. Taking the limit as $\delta\to 0$ completes the proof {of the first statement. As for the second statement, it follows directly from Theorem \ref{thm2.17} and Proposition \ref{prop2.15}.}}
\end{proof}

{\begin{cor}
		Assume $p\in C_0^1$ and $x=\eta_0(p)$. Then under {the} convergence $(\ref{4.2})$, the isometry $f:C_0\to C_0$ converges to an isometry $f_\infty:T_p C_0\to T_p C_0$. Moreover, $(T_x X,\mathsf{d}_{T_x X})$ is isometric to $(T_p C_0/f_\infty,\mathsf{d}_{T_p C_0}^\ast)$, where { $\mathsf{d}_{T_p C_0}^\ast$ denotes} the quotient metric induced by $f_\infty$.
\end{cor}
\begin{proof}
	The first statement is obvious. For the second one, let us take $p_\infty,q_\infty\in T_p C_0$ and two sequences $\{p_i\},\{q_i\}\subset C_0$ such that $p_i\to p_\infty$ and $q_i\to q_\infty$ respectively under (\ref{4.2}). 
	
	By Theorem \ref{thm4.10}, we may choose $q_i'\in {\eta_0}^{-1}(q_i)$ such that $\mathsf{d}_X(\eta_0(p_i),\eta_0(q_i))=\mathsf{d}(p_i,q_i')$. As $q_i'\to {\exists} q_\infty'\in T_p C_0$ and $\eta_0:({s_i}^{-1}C_0,p)\to ({s_i}^{-1}X,x)$ converges to $\eta_\infty:(T_p C_0,o_p)\to (T_x X,o_x)$ under (\ref{4.2}) {and (\ref{4.3})}, we see 
	\[
	\mathsf{d}_{T_p C_0}(p_\infty,q_\infty')=	\mathsf{d}_{T_p C_0}(f_\infty(p_\infty),f_\infty(q_\infty'))=\mathsf{d}_{T_x X}(\eta_\infty(p_\infty),\eta_\infty(q_\infty')).
	\]
	
	The proof is then completed by applying Proposition \ref{surjmapeta}.
	\end{proof}}

{We end this section by proving the following proposition. Note that Theorem \ref{thm1.3} is a direct consequence of this proposition.}
\begin{prop}\label{prop3.15}
  The number of components of $C_0$ is at most 2.
\end{prop}
\begin{proof}

 If $C_0^1\neq \emptyset$, let $p\in C_0^1$. For any $q\in C_0$, applying Theorem \ref{thm4.10} implies 
  \[
  \mathsf{d}(p,q)=\mathsf{d}(p,f(q))=\mathsf{d}_X\left(\eta_0(p),\eta_0(q)\right)\leqslant D.
  \]
  Therefore in this case $C_0$ is connected.

  If $C_0^1= \emptyset$ and $C_0$ is not connected, we fix a point $p\in C_0$. For any $q\in C_0$, applying Theorem \ref{thm4.10} again, we know that either $\mathsf{d}(p,q)=\mathsf{d}_X\left(\eta_0(p),\eta_0(q)\right)\leqslant D$ or $\mathsf{d}(p,f(q))=\mathsf{d}_X\left(\eta_0(p),\eta_0(q)\right)\leqslant D$. This yields that the number of components of $C_0$ is at most 2.
\end{proof}

\begin{cor}\label{lem4.1}
	{Assume $C_0^1\neq \emptyset$ and the dimension of $C_0$ is $k$.} Let $p\in C_0^1$ and $x=\eta_0(p)$. Then under {the} convergences $(\ref{4.1})$-$(\ref{4.3})$, {$\eta_\infty:T_p C\rightarrow T_x Y$ is surjective. In particular, we have $X_1\subset Y\setminus \mathcal R_{k+1}(Y)$.}

	
\end{cor}
\begin{proof}
	We first prove that 
	\begin{equation}\label{abcdefg11}
		\eta_\infty(T_p C)=T_x Y.
	\end{equation}
	
	For any $y_\infty\in T_x Y$, there exists a sequence of points $\{y_i\}\subset Y$ such that {$y_i\rightarrow y_\infty$} under {the} convergence (\ref{4.1}). By taking $x_i$ as the foot point of $y_i$ $(i\in \mathbb{N})$, we know that $\{x_i\}$ also converges to some $x_\infty\in T_x X$. 
	
	According to Proposition \ref{prop3.3}, we may take $p_i\in C_0$ such that $\eta(p_i,h_i)=y_i$ with $\mathsf{d}_Y(y_i, X)=h_i$. Then it follows from {Remark \ref{rmk3.3}}, {Theorems \ref{thmfffff4.5} and \ref{thm4.10}} that $\mathsf{d}_X(x_i,x)=\mathsf{d}(p_i,p)$. Applying Proposition \ref{prop3.5} we calculate that
	\[
	\begin{aligned}
		\frac{\mathsf{d}_C((p_i,h_i),p)}{s_i}&\leqslant \frac{\mathsf{d}_C((p_i,h_i),p_i)+\mathsf{d}_C(p_i,p)}{s_i}\\
		&\rightarrow \mathsf{d}_{T_x Y}(y_\infty,T_x X)+\mathsf{d}_{T_x X}(x_\infty,o_x)\ \ \text{as}\ i\rightarrow \infty.
	\end{aligned}
	\]
	Hence after passing to a subsequence, {we may assume} $(p_i,h_i)\rightarrow {\exists} q_\infty\in T_p C$ under {the} convergence (\ref{4.1}). In particular, we have $\eta_\infty(q_\infty)=y_\infty$, which completes the proof of (\ref{abcdefg11}).
	
	{As for the second statement, {by Remark \ref{rrmk4.2}, the dimension of $(Y,\mathsf{d}_Y)$ is $k+1$}. If $x\in X_1\cap \mathcal{R}_{k+1}(Y)$, then $T_x Y$ is isometric to $\mathbb{R}^{k+1}$. For any $R>0$, the local isometry property of $\eta_\infty|_{T_p C\setminus T_p C_0}$ implies that $\eta_\infty:B_R^{T_p C_0}(o_p)\times \{2R\} \rightarrow \mathbb{R}^{k}$ is an isometry. Since this holds for all $R$, applying Proposition \ref{prop4.1} then yields that $T_p C_0$ is isometric to $\mathbb{R}^k$ and $T_p C$ is isometric to the $(k+1)$-dimensional upper half plane in $\mathbb{R}^{k+1}_+$.

		However, since it is obvious that $\eta_\infty:T_p C\to \eta_\infty(T_p C)$ is an isometry, we deduce a contradiction with (\ref{abcdefg11}). }

\end{proof}

\subsection{The RCD$(K,n-1)$ structure of limit spaces}\label{sec5}

The main purpose of this section to introduce the RCD$(K,n-1)$ structure on $(X,\mathsf{d}_X)$. 

{When $C_0=C_0^2$ and $C_0$ is not connected,} by 
 Theorem \ref{thm4.10},
each component of $(C_0,\mathsf{d})$ is isometric to $(X,\mathsf{d}_X)$. By selecting any such component (still denoted as $(C_0,\mathsf{d},\mathfrak{m})$), we immediately see that $(X,\mathsf{d}_X,\mathfrak{m})$ becomes an RCD$(K,n-1)$ space. 

{When $C_0=C_0^1$, }$(C_0,\mathsf{d})$ is isometric to $(X,\mathsf{d}_X)$, making $(X,\mathsf{d}_X,\mathfrak{m})$ an RCD$(K,n-1)$ space as well. 

{We now turn to the case where both $C_0^1$ and $C_0^2$ are non-empty.} Here, Theorem \ref{thm4.10} directly implies the connectedness of $C_0$. {Refer to Examples \ref{exmp5.2} and \ref{rmk4.2} for examples.} 
{\begin{lem}\label{lem4.18}
	Assume both $C_0^1$ and $C_0^2$ are non-empty. Then the interior of $C_0^1$ is empty. Moreover, $X_2$ is geodesically convex.
\end{lem}
\begin{proof}
	The first statement follows directly from \cite[Lemma 4.26]{YZ19}. The second statement is a generalized version of \cite[Lemma 4.27]{YZ19}, we give an alternative proof here.
	
	Let $x,y\in X_2$ and $\{p_1,p_2\}={\eta_0}^{-1}(x)$. Let $\gamma:[0,l]\to X$ be a unit speed minimal geodesic from $x$ to $y$. 
	
	If $\gamma$ is not fully contained in $X_2$, then by Corollary \ref{cor4.15} there exists $l_0\in (0,l)$ which is the first parameter such that ${ \gamma(l_0)\in X_1}$. Take $\tilde{\tau}_1$, $\tilde{\tau}_2$ as unit speed minimal geodesics from $p_1$ to ${\eta_0}^{-1}(\gamma(l_0))$ and from $p_2$ to ${\eta_0}^{-1}(\gamma(l_0))$ respectively. The 1-Lipschitz property of $\eta_0$ together with Theorems \ref{thmfffff4.5} and \ref{thm4.10} implies $\eta_0(\tilde{\tau}_i)=\gamma|_{[0,l_0]}$ ($i=1,2$). Similarly, by taking $\tilde{\tau}$ as a unit speed minimal geodesic from ${\eta_0}^{-1}(\gamma(l_0))$ to a point $q\in {\eta_0}^{-1}(y)$, we know $\eta_0(\tilde{\tau})=\gamma|_{[l_0,l]}$.
	
	For $i=1,2$, define the curve $\tau_i:[0,l]\to C_0$ as
	\[
\tau_i(t)=\left\{\begin{aligned}
		\tilde{\tau}_i(t),\ \ &t\in [0,l_0],\\
		\tilde{\tau}(t),\ \ &t\in (l_0,l].
	\end{aligned}\right.
	\]It is obvious that $L^{\mathsf{d}}(\tau_i)=l$. However, since $\eta_0$ is 1-Lipschitz, one has 
	\[
	L^{\mathsf{d}}(\tau_i)\geqslant \mathsf{d}(p_i,q)\geqslant \mathsf{d}_X(x,y)=l.
	\] 
	Therefore, each $\tau_i$ is a minimal geodesic from $p_i$ to $q$, which contradicts Theorem \ref{thm2.17}.
\end{proof}
}{\begin{cor}[{\cite[Corollary 4.21]{YZ19}}]\label{cor4.19}
	Assume $C_0^2$ is non-empty. Then $\eta_0:C_0^2\to X_2$ is a locally isometric double covering space. 
\end{cor}
}

%
%
%
%
%
%
%
%

{The next} lemma is a direct consequence of \cite[Lemma 4.1]{GS19}.
 \begin{lem}\label{lem4.24}
	Assume both $C_0^1$ and $C_0^2$ are non-empty. Then $\mathfrak{m}(C_0^1)=0$.
\end{lem}

Provided $C_0^1$ has no contribution to the measure, under the assumption of the following {proposition}, there exists a natural RCD structure on $(X,\mathsf{d}_X)$. {Note that when $C_0^2$ is connected, there is an counterexample in Example \ref{fhaeiofhiaeohfae}.}

\begin{prop}\label{prop4.24}
	Assume both $C_0^1$ and $C_0^2$ are non-empty {and $C_0^2$ has two components.} Then there exists a Randon measure $\mathfrak{m}_X$ such that $(X,\mathsf{d}_X,\mathfrak{m}_X)$ is an $\mathrm{RCD}(K,n-1)$ space.
\end{prop}
{
\begin{proof}
	Fix $p\in C_0^2$ and let $x=\eta_0(p)$. Denote by $U$ the component of $C_0^2$ containing $p$. For any $y\in X_2$, Lemma \ref{lem4.18} ensures that any minimal geodesic $\gamma$ from $x$ to $y$ lies entirely in $\gamma\subset X_2$. By Corollary \ref{cor4.19}, ${\eta_0}^{-1}(\gamma)$ is also a minimal geodesic in {$U$}, connecting $p$ to a unique point $q\in {\eta_0}^{-1}(y)$ and satisfying $\mathsf{d}(p,q)=\mathsf{d}_X(\eta_0(p),\eta_0(q))$. This defines an isometry $\tau:X_2\to U\subset C_0^2$ via $\tau(y)=q$. 
	
	Next, observe from the proof of Lemma \ref{lem4.18} that the interior of $X_1$ is also empty. For any $y\in X_1$, take a sequence $\{y_i\}\subset X_2$ converging to $y$. The continuity of $\eta_0$ ensures $\tau(y_i)\to {\eta_0}^{-1}(y)$, allowing $\tau$ to extend to an isometry $\tau:X\to U\cup C_0^1$.

	Finally, letting $\mathfrak{m}_X=({\tau}^{-1})_\sharp \mathfrak{m}$ and using \cite[Proposition 7.2]{AMS16} and Lemma \ref{lem4.24}, we complete the proof.
	
\end{proof}
}

{

\begin{remark}
	If we let $\mathfrak{n}:=(\eta_0)_\sharp \mathfrak{m}$ be the push-forward measure of $\mathfrak{m}$ under $\eta_0$, it remains unknown whether $(X,\mathsf{d}_X,\mathfrak{n})$ is an $\mathrm{RCD}(K,n-1)$ space, for the following reason. {Although some of the following was {discussed} in \cite{GMMS18}, we give the details here for reader's convenience.}
	
	For any $g\in \mathrm{Lip}(X,\mathsf{d}_X)$, we define its lift $\tilde{g}:C_0\to\mathbb{R}$ as $\tilde{g}=g\circ\eta_0$. Theorem \ref{thm4.10} implies that {$\mathrm{lip}^{C_0}\,\tilde{g}=(\mathrm{lip}^{X}\,g)\circ\eta_0$.}

	Consider $g\in H^{1,2}(X,\mathsf{d}_X,\mathfrak{n})$ with $|\nabla^X g|$ denoting its minimal relaxed slope. By Remark \ref{rmk2.6}, there exists $\{g_i\}\subset \mathrm{Lip}(X,\mathsf{d}_X)\cap L^2(\mathfrak{n})$ such that $g_i\to g$ and $\mathrm{lip}^X\,g_i\to |\nabla^X g|$ in $L^2(\mathfrak{n})$. Therefore by Lemma \ref{lem4.24}, both $\{\tilde{g}_i\}$ and $\{\mathrm{lip}^{C_0}\,\tilde{g}_i\}$ are Cauchy sequences in $L^2(\mathfrak{m})$. Since $\{\tilde{g}_i\}\subset \mathrm{Lip}(C_0,\mathsf{d})$ satisfies $\tilde{g}_i\to \tilde{g}$ in $L^2(\mathfrak{m})$, Remark \ref{rmk2.11} combined with the completeness of Sobolev space yields $\mathrm{lip}^{C_0}\,\tilde{g}_i\to |\nabla^{C_0} \tilde{g}|$ in $L^2(\mathfrak{m})$. {After passing to a subsequence, for $\mathfrak{m}$-a.e. $p\in C_0$ (resp. $\mathfrak{n}$-a.e. $x\in X$) we have $\mathrm{lip}^{C_0}\,\tilde{g}_i(p)\to |\nabla^{C_0} \tilde{g}|(p)$ (resp. $\mathrm{lip}^X\,g_i(x)\to |\nabla^X g|(x)$)}, which then shows $|\nabla^{C_0} \tilde{g}|=|\nabla^X g|\circ \eta_0$ $\mathfrak{m}$-a.e. 
	
	The infinitesimally Hilbertian property of $(C_0,\mathsf{d},\mathfrak{m})$ implies that $(X,\mathsf{d}_X,\mathfrak{n})$ is also infinitesimally Hilbertian. Moreover, for any $g,h\in H^{1,2}(X,\mathsf{d}_X,\mathfrak{n})$, Remark \ref{rmk2.12} gives the equality 
	\[
	\langle\nabla^X g,\nabla^X h\rangle(\eta_0(p))=\langle\nabla^{C_0} \tilde{g},\nabla^{C_0} \tilde{h}\rangle(p) \text{ for $\mathfrak{m}$-a.e. }p\in C_0.
	\]
	
	However, for $g\in D(\Delta^X)\cap \mathrm{Lip}(X,\mathsf{d}_X)$, we can not verify that $\tilde{g}\in D(\Delta^{C_0})$. Therefore, we can not establish $\Delta^{C_0}\tilde{g}=(\Delta^X g)\circ\eta_0$ $\mathfrak{m}$-a.e. {unless we know $f$ preserves the measure $\mathfrak{m}$.} This obstruction prevents us from directly verifying the RCD$(K,n-1)$ condition for $(X,\mathsf{d}_X,\mathfrak{n})$ through Definition \ref{defn2.9}.  
\end{remark}
}
{Although in this case $(X,\mathsf{d}_X,\mathfrak{n})$ does not necessarily bear an RCD$(K,n-1)$ structure, it inherits rectifiablity from $(C_0,\mathsf{d},\mathfrak{m})$.
 \begin{prop}
	Let $k$ be the dimension of $C_0$. Then $(X,\mathsf{d}_X,\mathfrak{n})$ is $k$-rectifiable in the following sense: for any $\epsilon>0$, there exists a sequence $\{(G_i,\varphi_i)\}$, where $G_i\subset X$ are Borel sets satisfying $\mathfrak{n}(X\setminus \cup_i G_i)=0$ and the maps $\varphi_i:G_i\to \mathbb{R}^k$ are $(1+\epsilon)$-biLipschitz with their images.
\end{prop}
\begin{proof}
	Since the $k$-rectifiablity of $(C_0,\mathsf{d},\mathfrak{m})$ was established by Cheeger-Colding in \cite{ChCo3},  there exists a sequence $\{(U_i,\varphi_i)\}$, where $U_i\subset C_0$ are Borel sets satisfying $\mathfrak{m}(C_0\setminus \cup_i U_i)=0$ and the maps $\varphi_i:U_i\to \mathbb{R}^k$ are $(1+\epsilon)$-biLipschitz with their images. By Lemma \ref{lem4.24}, we may assume $U_i\subset C_0^2$ for all $i\in \mathbb{N}$.

	First, consider the case where $C_0^1=\emptyset$. Because {each component of} $C_0^2$ is compact, owing to Corollary \ref{cor4.19}, there exists a finite open covering $\{B_{r_i}(p_j)\}$ of $C_0^2$ such that $\eta_0:B_{r_j}(p_j)\to \eta_0(B_{r_j}(p_j))$ is an isometry for each $j$. By refining the sets $U_i$ via  $U_{i,j}=U_i\cap B_{r_j}(p_j)$ and relabelling, we may assume $\eta_0:U_i\to \eta_0(U_i)$ is an isometry for every $i$. 
	The desired conclusion follows by setting $(G_i,\psi_i)={ (\eta_0(U_i),\varphi_i\circ (\eta_0|_{U_i})^{-1})}$, ensuring each $\psi_i$ remains $(1+\epsilon)$-biLipschitz.
	
	For the case $C_0^1\neq\emptyset$, the same construction applies for $A_j:=\{p\in C_0^2:\mathsf{d}(p,C_0^1)\geqslant i^{-1}\}$ after invoking Lemma  \ref{lem4.24}, which thereby completes the proof.
\end{proof}
}
{The rest of this section is devoted to introduce an RCD$(K,n-1)$ structure to $(X,\mathsf{d}_X,\mathscr{H}^{ n-1})$ under the assumption that $C_0$ is a non-collapsed Ricci limit space.}

\begin{proof}[Proof of Theorem \ref{thm1.4}]
Combining (\ref{1.1}), (\ref{1.2}) with Theorem \ref{ncRiccilimit}{,} we know 
\[
\mathfrak{m}=\frac{\mathscr{H}_{\mathsf{d}}^{n-1}}{\mathscr{H}_{\mathsf{d}}^{n-1}(C_0)}.
\]
Then it follows from Proposition \ref{thmfffff4.5} that $f_\sharp \mathfrak{m}=\mathfrak{m}$. From Theorems \ref{thm2.18} and \ref{thm4.10} we see that $(X,\mathsf{d}_X,(\eta_0)_\sharp \mathfrak{m})$ is an RCD$(K,n-1)$ space. 

{When $C_0^2=\emptyset$, Theorem \ref{thm4.10} implies that $\eta_0:(C_0,\mathsf{d})\to (X,\mathsf{d}_X)$ is an isometry, from which the conclusion follows.

Therefore it suffices to consider the case where $C_0^2$ is non-empty. We claim that $\mathscr{H}^{n-1}_\mathsf{d}(C_0)(\eta_0)_\sharp \mathfrak{m}$ coincides with 2$\mathscr{H}^{n-1}_{\mathsf{d}_X}$.}

{By Lemma \ref{lem4.24}, we have $\mathscr{H}^{n-1}_\mathsf{d}(C_0^1)=0$. Since $\eta_0:(C_0,\mathsf{d})\to (X,\mathsf{d}_X)$ is 1-Lipschitz, {we get} $\mathscr{H}^{n-1}_{\mathsf{d}_X}(X_1)=0$.}

%
%
%
%

On the other hand, for any $p\in C_0^2$, by taking $\epsilon>0$ such that $B_\epsilon(p)\cap B_\epsilon(f(p))=\emptyset$, we know 
\[
\mathscr{H}^{n-1}_{\mathsf{d}}(B_\epsilon(p))=\mathscr{H}^{n-1}_{\mathsf{d}_X}\left(B_\epsilon (\eta_0(p))\right),
\]
which yields 
\[
\frac{\mathrm{d}(\eta_0)_\sharp(\mathfrak{m})}{\mathrm{d}\mathscr{H}^{n-1}_{\mathsf{d}_X}}=\frac{2}{\mathscr{H}^{n-1}_\mathsf{d}(C_0)}\ \ \text{on}\ X_2.
\]
This is enough to conclude.
\end{proof}
{
\begin{remark}
	Under (\ref{1.2}), the convergence (\ref{3.1}) is also non-collapsing. By Theorem \ref{ncRiccilimit} and Corollary \ref{lem4.1}, the Hausdorff dimension of $X_1$ with respect to $\mathsf{d}_Y$ is at most $n-2$. {It immediately follows from Proposition \ref{prop3.5} that the Hausdorff dimension of $X_1$ with respect to $\mathsf{d}_X$ is at most $n-2$. Moreover, by Theorem \ref{thm4.10}, we know the Hausdorff dimension} of $C_0^1$ with respect to $\mathsf{d}$ is at most $n-2$. Indeed, Example \ref{exmp5.2} shows that this upper bound is sharp.
	
	\end{remark}}

{By \cite[Proposition 2.10]{F87} and \cite[Theorem 1.9]{W08}, the family $\mathcal{M}(n,H,K,\lambda,D)$ is precompact in the renormalized measured Gromov-Hausdorff topology in the following sence: given any sequence of manifolds with boundary $\{(M_i,\mathrm{g}_i)\}\subset \mathcal{M}(n,H,K,\lambda,D)$, there exists a subsequence which is still denoted as $\{(M_i,\mathrm{g}_i)\}$, and a metric measure space $(Z,\mathsf{d}_Z,\mathfrak{m}_Z)$ such that
\[
\left(M_i,\mathsf{d}_{\mathrm{g}_i},\frac{\mathrm{vol}_{\mathrm{g}_i}}{\mathrm{vol}_{\mathrm{g}_i}(M_i)}\right)\xrightarrow{\mathrm{mGH}}(Z,\mathsf{d}_Z,\mathfrak{m}_Z).
\]

If moreover $\{(M_i,\mathrm{g}_i)\}$ inradius collapses, then by Theorem \ref{thm3.7} and writing $\mathfrak{m}_Z$ as $\mathfrak{m}_X$, we may assume 
\begin{align}\label{ffaeohof}
\left(M_i,\mathsf{d}_{\mathrm{g}_i},\frac{\mathrm{vol}_{\mathrm{g}_i}}{\mathrm{vol}_{\mathrm{g}_i}(M_i)}\right)\xrightarrow{\mathrm{mGH}}(X,\mathsf{d}_X,\mathfrak{m}_X).
\end{align}

We first show that $(X,\mathsf{d}_X)$ bears the weak infinitesimal splitting property, which is defined as follows.
\begin{defn}
	A metric space $(X,\mathsf{d}_X)$ is said to have the \textit{weak infinitesimal splitting property} if, for every $x \in X $, there exists a sequence $s_i\downarrow 0$ such that $({s_i}^{-1}X,x)$ pGH converges to a metric space $(T_x X, o_x)$ which satisfies the splitting property: whenever $\gamma \subset T_x X$ is a geodesic line, there exists a metric space $\Omega$ such that $T_x X$ is isometric to $ \Omega\times\mathbb{R}$ with $\pi(\gamma)$ being a single point $\omega\in\Omega$, where $\pi$ is the projection onto the first coordinate. 
\end{defn}
\begin{thm}\label{thm4.30}
	The metric space $(X,\mathsf{d}_X)$ {obtained} in (\ref{ffaeohof}) satisfies the weak infinitesimal splitting property. 
\end{thm}
\begin{proof}
By Proposition \ref{lem4.4}, to show the weakly infinitesimal splitting property, it suffices to consider the case where $x\in X_1$ and $T_x X$ contains a line $\gamma$. 

Theorem \ref{CCsplit} ensures the existence of a metric space $(\Omega,\mathsf{d}_\Omega)$ such that $(T_x Y,\mathsf{d}_{T_x Y})$ is isometric to $(\Omega\times \mathbb{R},\sqrt{{\mathsf{d}_\Omega}^2+{\mathsf{d}_\mathbb{R}}^2})$ with $\pi(\gamma)=\omega\in \Omega$. 

Consider $x_\infty=(\omega',t')\in T_x X$ with $\omega'\neq \omega$. For each $t\in \mathbb{R}$, let $\gamma_t$ be a unit speed minimal geodesic from $x_\infty$ to $(\omega,t)$. By Lemma \ref{convex}, each $\gamma_t$ lies entirely in $T_x X$. The product structure of $T_x Y$ implies that as $t\to\infty$, $\gamma_t$ converges to the ray $v_+:[t',\infty)\to X$ defined by $v_+(s)=(\omega',s)$. Similarly, we obtain the ray $v_-:(-\infty,t']\to X$ with $v_-(s)=(\omega',s)$.

Consequently, the line $v:s\mapsto (\omega',s)$ is parallel to $\gamma$ and entirely contained in $T_x X$. Applying Lemma 
\ref{convex} once more, we conclude that $(T_x X,\mathsf{d}_{T_x X})$ is isometric to $(\Gamma\times\mathbb{R},\sqrt{{\mathsf{d}_\Omega}^2+{\mathsf{d}_\mathbb{R}}^2})$ for some subset $\Gamma\subset \Omega$, completing the proof.
\end{proof}

Using \cite[Theorem 1.2]{NPS25}, we immediately get the following proposition.
\begin{prop}\label{prop4.31}
	The metric space $(X,\mathsf{d}_X)$ defined in (\ref{ffaeohof}) is universally infinitesimally Hilbertian. In particular, $(X,\mathsf{d}_X,\mathfrak{m}_X)$ is infinitesimally Hilbertian.
\end{prop}}
{When $\lambda=0$, the following result holds. As demonstrated in {Example \ref{exmp5.2}}, the reference measure $\mathfrak{m}_X$ may not coincide with the $(n-1)$-dimensional Hausdorff measure of $(X,\mathsf{d}_X)$. {Moreover, from Example \ref{eexmp5.4} we know if $\lambda>0$, then in general $(X,\mathsf{d}_X,\mathfrak{m}_X)$ may not satisfy some RCD condition.}
\begin{thm}
Assume $\lambda=0$. {Then $(X,\mathsf{d}_X,\mathfrak{m}_X)$ is an $\mathrm{RCD}(K,n)$ space. If moreover}  $(C_0,\mathsf{d})$ is a non-collapsed Ricci limit space of $(\partial M_i,\mathrm{g}_{\partial M_i})$ $($i.e. (\ref{1.2}) holds$)$, {then we have} $\mathfrak{m}_X\ll \mathscr{H}^{n-1}$. 
\end{thm}
\begin{proof}
The first statement follows directly from Theorems \ref{thm2.15} and \ref{thm2.16}. {If (\ref{1.2}) holds,} Theorem \ref{thm1.4} and Definition \ref{defn2.9} yield that $(X,\mathsf{d}_X,\mathscr{H}^{n-1})$ is an RCD$(H,n)$ space. \cite[Remark 4.3]{K19} shows the dimension of $(X,\mathsf{d}_X,\mathfrak{m}_X)$ is also $n-1$, which combined with \cite[Theorem 3.8]{BS20} gives 
\[
\mathfrak{m}_X(X\setminus \mathcal{R}_{n-1}(X))=0.
\]

\cite[Theorem 4.1]{AHT18} guarantees the existence of a subset $\mathcal{R}_{n-1}^\ast(X)\subset \mathcal{R}_{n-1}{(X)}$ such that 
\[\mathfrak{m}_X(\mathcal{R}_{n-1}(X)\setminus\mathcal{R}_{n-1}^\ast(X))=0,
\]
and that for $\mathfrak{m}$-a.e. $x\in \mathcal{R}_{n-1}^\ast(X)$, the density limit $\lim_{r\to 0}r^{-(n-1)}\mathfrak{m}_X(B_r(x))$ exists and is positive and finite.

For any Borel set $A\subset X$ with $\mathscr{H}^{n-1}(A)=0$, define
\[
A_i:=\left\{y\in A\cap \mathcal{R}_{n-1}^\ast(X)\Big{|}\lim_{r\to 0}\frac{\mathfrak{m}_X(B_r(x))}{r^{n-1}}\leqslant i\right\}, i\in\mathbb{N}.
\]
Applying \cite[Theorem 2.4.3]{AT04} shows $\mathfrak{m}_X(A_i)=0$ for each $i$. The conclusion follows since
\[
\mathfrak{m}_X(A)=\mathfrak{m}_X\left((A\setminus\mathcal{R}_{n-1}^\ast(X))\cup_{i=1}^\infty A_i\right)=0.
\]  
\end{proof}}

{\begin{remark}\label{rmk4.33}
		In summary, there exists a Randon measure $\mu$ on $X$ with $\mathrm{supp}(\mu)=X$ such that $(X,\mathsf{d}_X,\mu)$ is an RCD$(K,n)$ space, except in the following cases:
{ \begin{itemize}
	\item[$(1)$] 	$\lambda>0$, $\lim_{i\to\infty}\mathrm{vol}_{\mathrm{g}_{\partial M_i}}({\partial M_i})=0$ and $C_0=C_0^2$  is connected.
	\item[$(2)$]   $\lambda>0$, $\lim_{i\to\infty}\mathrm{vol}_{\mathrm{g}_{\partial M_i}}({\partial M_i})=0$, both $C_0^1$ and $C_0^2$ are non-empty and $C_0^2$ is connected.
\end{itemize}
In either of the above cases,} whether such a Radon measure exists remains unknown.
\end{remark}}

{Let us also consider the special case where the limit space of inradius collapsed manifolds in $\mathcal{M}(n,H,K,\lambda,D)$ is a smooth Riemannian manifold. Compare \cite[Theorem 5.4]{YZ19}.}  
{
\begin{thm}\label{prop4.33}
	Assume $(X,\mathsf{d}_X)$ is an $(n-1)$-dimensional smooth closed Riemannian manifold. Then for any sufficiently large $i$, $M_i$ is diffeomorphic to {$X\times I$} or a twisted product $X\tilde{\times}\, I $, where $I=[-1,1]$. 
\end{thm}
\begin{proof}
We first show $X_2\neq \emptyset$.

Assume $X=X_1$. Then by Theorem \ref{thm4.10}, we have $C_0=C_0^1$ and $(C_0,\mathsf{d})$ is isometric to $(X,\mathsf{d}_X)$. According to \cite[Theorem A.1.12]{ChCo1}, for all sufficiently large $i$, $\partial M_i$ is diffeomorphic to $C_0$,  and the warped cylinder $C_{M_i}$ is diffeomorphic to $C=C_0\times_\phi[0,t_0]$. It follows that the boundary of $C_{M_i}$ has exactly two components.

On the other hand, observe that the glued space $\widetilde{M}_i$ has only one boundary component for all sufficiently large $i$. Since $(Y,\mathsf{d}_Y)$ is now isometric to $(C,\mathsf{d}_C)$, applying \cite[Theorem A.1.12]{ChCo1} again yields {that $\widetilde{M}_i$ is diffeomorphic to $Y$ for sufficiently large $i$, implying} the topological of $Y=C$ must also have one component. A contradiction. 

Suppose there exists a point $x\in X_1$. We take a geodesic convex neighborhood $U$ near $x$ whose diameter is less than the injectivity radius of $X$. By Lemma \ref{lem4.18}, one can first take $y\in U\cap X_2$ sufficiently close to {$x$} and extend the geodesic $\gamma_{y,x}$ to geodesic $\gamma_{y,z}$ for some $z\in U$. This is impossible due to the proof of Lemma \ref{lem4.18}. Hence $X=X_2$.

Using \cite[Theorem A.1.12]{ChCo1} again, we see for sufficiently large $i$, $\widetilde{M}_i$ is diffeomorphic to $Y$. Since $Y$ is diffeomorphic to either $X\times I$ or $C_0\times I/ (x,t) \sim (f(x),-t)$, which is a twisted $I$-{product} over $X$, the conclusion follows.

\end{proof}}
\subsection{The unbounded diameter case}

{
In this subsection, we introduce the unbounded diameter case of inradius collapsed manifolds in connection with \cite[Section 6]{YZ19}. Since the proof of the results are almost the same, we omit some details.

		Let $\{(M_i,\mathrm{g}_i,p_i)\}$ be a sequence of manifolds with boundary in $\mathcal{M}(n,H,K,\lambda)$ such that $p_i\in\partial M_i$ for all $i\in \mathbb{N}$ and that $\mathrm{inrad}(M_i)\to 0$ as $i\to\infty$.

As in the beginning of Section \ref{sec3}, we still attach to each $M_i$ a warped cylinder $C_{M_i}:=M_i\times_\phi [0,t_0]$, obtaining a pointed metric space $(\widetilde{M}_i,\mathsf{d}_{\widetilde{M}_i},p_i)$. Because each $(\widetilde{M}_i,\mathsf{d}_{\widetilde{M}_i},(\mathscr{H}^n(B_1(p_i)))^{-1}\mathscr{H}^n,p_i)$ is an RCD$(-J(n,H,K,\lambda,\varepsilon_0,t_0),n)$ space, after passing to a subsequence, we may assume $(\widetilde{M}_i,\mathsf{d}_{\widetilde{M}_i},p_i)$ converge to $(Y,\mathsf{d}_Y,{x})$ in the {pointed} Gromov-Hausdorff sense. Viewing the convergence of subsets, we also assume $\{(M_i,\mathsf{d}_{\widetilde{M}_i},{p_i})\}$ pGH converges to $(X,\mathsf{d}_Y,x)$ as $i\to \infty$. 

Let
\[
(C_{t_0}^Y,\tilde{\mathsf{d}}_{t_0}):=(\{y\in Y|\mathsf{d}_Y(y,X)=t_0\},\mathsf{d}_Y),
\]
which is considered component-wisely. We introduce the metrics $\mathsf{d}_X$ on $X$ and $\mathsf{d}_{t_0}$ on $C_{t_0}^Y$ as the intrinsic metrics induced by $\mathsf{d}_Y$ and $\tilde{\mathsf{d}}_{t_0}$ respectively. We set 
\[
\ (C_0,\mathsf{d}):=(C_{t_0}^Y\times\{0\},(\phi(t_0))^{-1}{ \mathsf{d}_{t_0}}),\ C:=C_0\times_\phi[0,t_0].
\]

Notice that for each $p\in C_0$, one can identify an element $y\in C_{t_0}^Y$, denoted by $\eta(p,t_0)$. Moreover, there exists a unique perpendicular $\gamma_{\eta(p,t_0)}:[0,t_0]\to Y$ to $X$ such that ${\gamma_{\eta(p,t_0)}(t_0)=\eta(p,t_0)}$. Then we define the surjective 1-Lipschitz maps $\eta:C\to Y$ and $\eta_0:C_0\to X$, respectively, by 
\[
\eta(p,t):={\gamma_{\eta(p,t_0)}(t)},\ \eta_0(p)=\eta(p,0).
\] 
It is clear that Proposition \ref{prop3.3}, Proposition \ref{prop3.5}, Theorem \ref{thm3.7} and Lemma \ref{lem3.8} continue to hold in this setting. Therefore, we may define the subsets as in Definition \ref{11111defn3.9} and the involution $f:C_0\to C_0$ as in (\ref{feaohfoiaehfiohae}).

Since the local isometry property of $\eta:Y\setminus X\to C\setminus C_0$, together with Theorem \ref{CCsplit}, guarantees the existence of the tangent space $T_p C_0$ for every $p\in C_0$, the local arguments in {Subsections \ref{sec4.1}-\ref{sec5}} remain valid, and all the results continue to hold for the present noncompact $C_0$. In particular, we know the metric space $(X,\mathsf{d}_X)$ is isometric to $(C_0,\mathsf{d})/f$ and the number of components of $C_0$ is at most two. {This proves Theorem \ref{thm1.10}(1). Arguing similarly as in Theorem \ref{thm4.10} and Proposition \ref{prop4.31}, we obtain Theorem \ref{thm1.10}(2). The proof of Theorem \ref{thm1.10}(3) follows from the same reasoning as that of Theorem \ref{thm1.4}.} 

{Next we prove Theorem \ref{thm1.11}, which follows from the same monodromy argument as \cite{YZ19}. We briefly outline the ideas for reader's convenience.}

Define $\widetilde{\mathcal{M}}\mathcal{M}(n,H,K,\lambda)$ as the set of all triples $(\widetilde{M},M,p)$, where $(M,\mathrm{g})\in \mathcal{M}(n,H,K,\lambda)$, $p\in \partial M$, and $\widetilde{M}$ is obtained via the extension procedure. Let $\partial_0\widetilde{\mathcal{M}}\mathcal{M}(n,H,K,\lambda)$ be the set of all pointed Gromov-Hausdorff limit spaces $(Y,X,x)$ of sequences $\{(\widetilde{M}_i,M_i,p_i)\}\subset {\widetilde{\mathcal{M}}}\mathcal{M}(n,H,K,\lambda)$ satisfying $\lim_{i\to \infty}\mathrm{inrad}(M_i)=0$.

Now for every $\delta>0$, there exists $\epsilon=\epsilon(\delta)>0$ such that for every $(M,\mathrm{g},p)\in \mathcal{M}(n,H,K,\lambda)$ with $\mathrm{inrad}(M_i)<\epsilon$, one can find $(Y,X,x)\in\partial_0\widetilde{\mathcal{M}}\mathcal{M}(n,H,K,\lambda)$ satisfying the following properties.
\begin{itemize}
	\item $\mathsf{d}_{\mathrm{pGH}}((\widetilde{M},\mathsf{d}_{\widetilde{M}},p),(Y,\mathsf{d}_Y,x))<\delta$.
	\item $\mathsf{d}_{\mathrm{pGH}}((B_{1/\delta}^{(M,\mathsf{d}_{\widetilde{M}})}{(p)},\mathsf{d}_{\widetilde{M}},p),(B^{(X,\mathsf{d}_Y)}_{1/\delta}(x),\mathsf{d}_Y,x))<\delta$.
	\item $\mathsf{d}_{\mathrm{pGH}}((\partial M\cap B^{\widetilde{M}}_{1/\delta}{(p)},\mathsf{d}_{\mathrm{g}_{\partial M}},p),(C_0\cap B^{C}_{1/\delta}({\tilde{x}}),\mathsf{d},{\tilde{x}}))<\delta$. Here { $\tilde{x}\in {\eta_0}^{-1}(x)$} and the pGH distance can be measured component-wisely.
\end{itemize}
Such space is called a {\textit{$\delta$-limit space}} of $(\widetilde{M},M,p)$.

Fix a sufficiently large $R\gg t_0$. First, choose a sufficiently small $\delta=\delta(n,H,K,\lambda,R)>0$ and take $\epsilon=\epsilon(\delta)$ such that for any $(M,\mathrm{g})\in \mathcal{M}(n,H,K,\lambda)$ with $\mathrm{inrad}(M)<\epsilon$, the following holds: if no $\delta$-limit space $(Y,X,x)$ of $(\widetilde{M},M,p)$ has the property that every point $z\in X\cap B^Y_{16R}(x)$ satisfies $\#{(\eta_0)}^{-1}(z)=2$ (a property we will henceforth call { \textit{double in scale $16R$}} for convenience), then any $q\in \partial M\cap B_{8R}^{\widetilde{M}}(p)$ can be connected to $p$ by a continuous curve lying in $\partial M$. Otherwise, $\partial M\cap B_{8R}^{\widetilde{M}}(p)$ has two components, and there exists a point $p'\in \partial M\cap B_{8R}^{\widetilde{M}}(p)$ such that any point $q\in \partial M\cap B_{8R}^{\widetilde{M}}(p)$ can be connected by a continuous curve in $\partial M$ either to $p$ or $p'$. 

Therefore, if $\partial M$ is disconnected, we claim that for every point $p\in \partial M$, any $\delta$-limit space $(Y,X,x)$ of $(\widetilde{M},M,p)$ is double in scale $16R$. 

Assume there exists $p\in \partial M$ such that some $\delta$-limit space is not double. Take a point $q$ in a different component of $\partial M$ {with $\mathsf{d}_\mathrm{g}(p,q)=l>100R$} and let $\gamma:[0,l]\to M$ be a unit speed minimal geodesic from $p$ to $q$ in $M$. For $\gamma(R)$, one can choose a nearest point $p_1\in \partial M$ with $\mathsf{d}_{\widetilde{M}}(p,p_1)\leqslant\mathsf{d}_\mathrm{g}(p,p_1)\leqslant 2R$. As a result, any $\delta$-limit space of $(\widetilde{M},M,p_1)$ in scale $16R$ is not double, implying that $\partial M\cap B_{8R}^{\widetilde{M}}(p_1)$ is connected. Applying a monodromy argument, this property also holds for $\gamma(kR)$ with $k\leqslant [l/R]$, ultimately leading to a contradiction.

We remark that, via a similar monodromy argument, one can show that for $\partial M$ for $(M,\mathrm{g})\in \mathcal{M}(n,H,K,\lambda)$ with $\mathrm{inrad}(M)<\epsilon$, the number of components of $\partial M$ is at most two. {This completes the proof of Theorem \ref{thm1.11}.}
%
%
%
%
%

}

\section{Examples}\label{section5}

We present several {typical examples of inradius collpased manifolds} in this section. {The first example illustrates why the additional assumption of a lower Ricci curvature bound on the boundaries is necessary.}

\begin{exmp}\label{ex5.6}
	
	{{Let $n\geqslant 3$. For each $i\in \mathbb{N}$ we take $B_i$ as the unit ball in the space form of constant curvature $-i/(n-2)$ and glue $B_i$ with itself along $\partial B_i$. After smoothing out near $\partial B_i$ (Theorem \ref{thm2.1}), we obtain a sequence of $(n-1)$-dimensional compact Riemannian manifolds without boundary $\{(N_i,\mathrm{h}_i)\}$ such that $\mathrm{diam}(N_i)\in (1,3)$ and that $-(i+1)\leqslant \mathrm{Ric}_{\mathrm{h}_i}\leqslant -(i-1)$. }
		
		Let us consider the warped products $(M_i,\mathrm{g}_i)=\left(N_i\times_{f_i} [0,1/i],\mathrm{d}t^2+{f_i}^2 \mathrm{h}_i\right)$, where $f_i(t)=-2it^2+1$ ($i\geqslant 2$). Then $\{(M_i,\mathrm{g}_i)\}$ has a uniform lower Ricci curvature bound and a uniform two sided second fundamental form bound. Moreover, we have $\mathrm{inrad}(M_i)\rightarrow 0$ as $i\rightarrow { \infty}$. However, there does not exists a real number $K\in \mathbb{R}$ such that $\mathrm{Ric}_{\mathrm{h}_i}\geqslant K$ for all $i$.}
\end{exmp}

	{ In the remainder of this section, we denote by {$\mathbb{S}^n(r)$ the $n$-dimensional sphere with radius $r$ and by $\mathbb{S}^n=\mathbb{S}^n(1)$ for brevity.}}
{ \begin{exmp}\label{fhaeiofhiaeohfae}
	Let $\varphi:\mathbb{S}^2\to \mathbb{S}^2$ be the involution defined as
	$\varphi(x,y,z)=(-x,-y,z)$. For $i\in \mathbb{N}$, let $(N_i,h_i)$ be the product Riemannian manifold defined by $N_i:=\mathbb{S}^2\times\mathbb{S}^1(i^{-1})\times[-i^{-1},i^{-1}]$. We also define the map $\Phi_i:N_i\to N_i$ as $\Phi((x,y,z),\theta,t)=(\varphi(x,y,z),-\theta,-t)$. 
	
	By taking $(M_i,\mathrm{g}_i)$ as the quotient Riemannian manifold $(N_i,\mathrm{h}_i)/\Phi_i$, we know as $i\to\infty$, $(M_i,\mathsf{d}_{\mathrm{g}_i})$ GH converges to $\mathbb{S}^2/\varphi$ and $(\partial M_i,\mathsf{d}_{\mathrm{g}_{\partial M_i}})$ GH converges to $\mathbb{S}^2$. However, there does not exist an isometric embedding from $\mathbb{S}^2/\varphi$ to $\mathbb{S}^2(r)$ for { any} $r>0$. { Compare Proposition \ref{prop4.24}.}
\end{exmp}}


{We give an example of inradius collapsed manifolds whose boundaries are non-collapsed.  
 \begin{exmp}\label{exmp5.2}
{	For $i\in \mathbb{N}$, let 
	\[
	{ (N_i,\mathrm{h}_i)}=\left(\left\{(x,y,z)\in \mathbb{R}^3|\,x^2+i^2(y^2+z^2)=1,z\leqslant 0\right\},{ \mathrm{h}_i}\right),
	\]
where {$\mathrm{h}_i$} is the natural Riemannian metric induced by $\mathrm{g}_{\mathbb{R}^3}$. {Let $(M_i,\mathrm{g}_i)$ be the product of $N_i$ and $\mathbb{S}^1$, i.e. $M_i=N_i\times\mathbb{S}^1$ and $\mathrm{g}_i=\mathrm{h}_i+\mathrm{g}_{\mathbb{S}^1}$.} 
	It is clear that 
	\[
	\mathrm{Ric}_{\mathrm{{ g}}_i}\geqslant 0, \ 	\mathrm{Ric}_{\mathrm{{ g}}_{\partial { M_i}}}\geqslant 0,\ \text{\Rmnum{2}}_{\partial { M_i}}=0.
	\] }
As $i\to\infty$, we have
	\[
	\left(\partial M_i,\mathrm{g}_{\partial M_i},
	\mathrm{vol}_{\mathrm{g}_{\partial M_i}}\right)
	\xrightarrow{\mathrm{mGH}}(C_0,\mathsf{d},\mathfrak{m}):={\left(\mathbb{S}^1\left(2/\pi\right)\times \mathbb{S}^1,\mathsf{d}_{\mathbb{S}^1(2/\pi)\times\mathbb{S}^1},\mathscr{H}^2\right)},
	\]
	and
	\[
	\left(M_i,\mathsf{d}_{\mathrm{g}_i},\frac{\mathrm{vol}_{\mathrm{g}_i}}{\mathrm{vol}_{\mathrm{g}_i}(M_i)}\right)\xrightarrow{\mathrm{mGH}}(X,\mathsf{d}_X,\mathfrak{m}_X):={\left([-1,1]\times\mathbb{S}^1,\mathsf{d}_{[-1,1]\times \mathbb{S}^1},\frac{3(1-t^2)}{4}\mathscr{H}^2\right)},
	\]
{where $\mathsf{d}_{\mathbb{S}^1(2/\pi)\times\mathbb{S}^1}$ and $\mathsf{d}_{[-1,1]\times \mathbb{S}^1}$ are the distances induced by standard product structures, and $t$ is the coordinate on $[-1,1]$. Moreover, the map $\eta_0:C_0\to X$ is given by
\[
\eta_0\left(\frac{2}{\pi}\exp(\sqrt{-1}\,\alpha),\exp(\sqrt{-1}\,\beta)\right)=\left(-1+\frac{2|\alpha-\pi|}{\pi},\exp(\sqrt{-1}\beta)\right),\ \forall \alpha,\beta\in [0,2\pi).
\]
It is obvious that $X_1=\{\{\pm 1\}\times \mathbb{S}^1\}$ and $C_0^1={\eta_0}^{-1}(X_1)$, and that these sets both have Hausdorff dimension 1.}
\end{exmp}
}
%
%
The following example shows that if we release the assumption (\ref{1.2}), $f$ may not {preserve the measure $\mathfrak{m}$}. 
\begin{exmp}\label{rmk4.2}
{ For any two sufficiently large number $t_1,t_2$,} we will construct a sequence of inradius collapsed manifolds $(M_i,\mathrm{g}_i)$ such that 
\begin{equation}\label{aaaaa5.1}
\left(\partial M_i,\mathsf{d}_{\mathrm{g}_{\partial M_i}},\frac{\mathrm{vol}_{\mathrm{g}_{\partial M_i}}}{\mathrm{vol}_{\mathrm{g}_{\partial M_i}}(\partial M_i)}\right)\xrightarrow{\mathrm{mGH}}  \left({\mathbb{S}^1{ \left(2/\pi\right)},\mathsf{d}_{ \mathrm{g}_{\mathbb{S}^1(2/\pi)}},\mathfrak{m}}\right),
\end{equation}
and that the measure $\mathfrak{m}$ satisfies
{
    \begin{equation}\label{eqnfjlajfleajflea}
\frac{\mathfrak{m}(U)}{\mathfrak{m}(V)} \geqslant\frac{t_1}{4t_2}.
  \end{equation}}{ Here $U,V$ are defined as
\[
U:=\left\{\frac{2}{\pi}\,{ \exp(\sqrt{-1}\, \alpha)}\Big{|}\,\alpha\in \left[\frac{\pi}{2}\cos\left(\frac{\pi}{4} +\epsilon\right),\frac{\pi}{2}\cos\left(\frac{\pi}{4}-\epsilon\right)\right]\right\};
\]
\[
V:=\left\{\frac{2}{\pi}\,{ \exp(-\sqrt{-1}\, \alpha)}\Big{|}\,\alpha\in \left[\frac{\pi}{2}\cos\left(\frac{\pi}{4} +\epsilon\right),\frac{\pi}{2}\cos\left(\frac{\pi}{4}-\epsilon\right)\right]\right\}.
\]}

Let $M_i=\mathbb{S}^1\times N_i$, where 
\[
{ N_i:=\left\{(x,y,z)\in \mathbb{R}^3|\,\mathop{x^2+i^2(y^2+z^2)=1,z\geqslant 0}\right\}},
\]
and $\mathrm{g}_{N_i}$ be the natural Riemannian metric induced by $\mathrm{g}_{\mathbb{R}^3}$.  

For $ \mu \in {\left[0,\frac{\pi}{2}\right]}$, $\theta\in [0,2\pi)$, take the coordinate chart on $N_i$ as 
\[
{ \left\{
  \begin{aligned}
   x= \sin \mu \cos \theta,\ \ \ \ \\
 y=i^{-1}\sin \mu \sin \theta,\\
z=i^{-1} \cos \mu,\ \ \ \ \ \ 
  \end{aligned}
  \right.}
\]
We define the Riemannian metric $\mathrm{g}_i$ on $M_i$ as
  \begin{equation}\label{5555.3}
 \mathrm{g}_i:= { i^{-2}}\varphi^2(\theta)\,{ \mathrm{g}_{\mathbb{S}^1}}+\mathrm{g}_{N_i},
  \end{equation}
where 
\[
\begin{aligned}
  \varphi:[0,2\pi]&\longrightarrow \mathbb{R}\\
              \theta&\longmapsto
  \left\{
  \begin{aligned}
    t_1 \exp\left(\frac{1}{\left[\theta-\left(\frac{\pi}{4}-\epsilon\right)\right]\left[\theta-\left(\frac{\pi}{4}+\epsilon\right)\right]}\right)+1,\ \ &\  \text{if}\ \theta\in \left[\frac{\pi}{4}-\epsilon,\frac{\pi}{4}+\epsilon\right];\\
      t_2 \exp\left(\frac{1}{\left[\theta-\left(\frac{7\pi}{4}-\epsilon\right)\right]\left[\theta-\left(\frac{7\pi}{4}+\epsilon\right)\right]}\right)+1,&\  \text{if}\ \theta\in \left[\frac{7\pi}{4}-\epsilon,\frac{7\pi}{4}+\epsilon\right];\\
      1, \ \ \ \ \ \ \ \ \ \ \ \ \ \ \ \ \ \ \ \ \ \ \ \ \ \ \ \ \ \ \ \ \ \ \ \ \ \ \ \ \ \ \ \ \ \ \ \ \ \ \ \ \   &\  \text{otherwise},
  \end{aligned}
  \right.
\end{aligned}
  \]
with $t_1>t_2>0$ and $\epsilon\in { (}0,\pi /8]$ { to be determined later}.

 Then we have
{\begin{equation}\label{sosorry5.3}
\mathrm{Ric}_{\mathrm{g}_i}\geqslant -J \ \text{and}\ \mathrm{Ric}_{\mathrm{g}_{\partial M_i}}\geqslant -J\ \ \text{for some constant }J=J(\epsilon,t_1,t_2).
\end{equation}
}
{Let
\[ 
U_i:=\left\{(x,y,z)\in N_i| \mathop{\theta\in  \left[\frac{\pi}{4}-\epsilon,\frac{\pi}{4}+\epsilon\right], { \mu=\frac{\pi}{2}}}\right\},
\]
\[
V_i:=\left\{ (x,y,z)\in N_i| \mathop{\theta\in \left[\frac{7\pi}{4}-\epsilon,\frac{7\pi}{4}+\epsilon\right], { \mu=\frac{\pi}{2}}}\right\}.
\] 

Since we can control the value of cosine on $\left[\frac{\pi}{4}-\epsilon,\frac{\pi}{4}+\epsilon\right]$ and $\left[\frac{7\pi}{4}-\epsilon,\frac{7\pi}{4}+\epsilon\right]$, we calculate that for sufficiently large $i$,
{\begin{equation}\label{sorry5.2}
\frac{\mathrm{vol}_{\mathrm{g}_{\partial M_i}}(\mathbb{S}^1\times U_i)}{\mathrm{vol}_{\mathrm{g}_{\partial M_i}}(\mathbb{S}^1\times V_i)}=\frac{\int_{\frac{\pi}{4}-\epsilon}^{\frac{\pi}{4}+\epsilon} \varphi(\theta)\sqrt{\cos^2\theta+{ i^{-2}} \sin^2 \theta}\mathop{\mathrm{d}\theta}}{\int_{\frac{7\pi}{4}-\epsilon}^{\frac{7\pi}{4}+\epsilon} \varphi(\theta)\sqrt{\cos^2\theta+{ i^{-2}} \sin^2 \theta}\mathop{\mathrm{d}\theta}}\geqslant { \Psi}(\epsilon) \frac{\int_{\frac{\pi}{4}-\epsilon}^{\frac{\pi}{4}+\epsilon} \varphi(\theta)\mathop{\mathrm{d}\theta}}{\int_{\frac{7\pi}{4}-\epsilon}^{\frac{7\pi}{4}+\epsilon} \varphi(\theta)\mathop{\mathrm{d}\theta}},
\end{equation}}
{ where $\Psi(\epsilon)$ is a constant depending on $\epsilon$ with $\lim_{\epsilon\to 0}\Psi(\epsilon)=1$.}

If we take $t_1,t_2$ to be sufficiently large, then the last term of (\ref{sorry5.2}) can be sufficiently close to ${\Psi}(\epsilon) t_2/t_1$. Therefore we know for any two sufficiently large $t_1,t_2$ it holds that
\begin{equation}\label{imsorry5.4}
\frac{\mathrm{vol}_{\mathrm{g}_{\partial M_i}}(\mathbb{S}^1\times U_i)}{\mathrm{vol}_{\mathrm{g}_{\partial M_i}}(\mathbb{S}^1\times V_i)}\geqslant \frac{{\Psi}(\epsilon)t_1}{2t_2}.
\end{equation}
Let us just fix a small $\epsilon$ and rewrite (\ref{imsorry5.4}) as
\[
\frac{\mathrm{vol}_{\mathrm{g}_{\partial M_i}}(\mathbb{S}^1\times U_i)}{\mathrm{vol}_{\mathrm{g}_{\partial M_i}}(\mathbb{S}^1\times V_i)}\geqslant \frac{t_1}{4t_2}.
\]

}

According to (\ref{sosorry5.3}), we have mGH convergence (\ref{aaaaa5.1}). Moreover, {because $U_i,V_i$ converge to $U,V$ respectively under (\ref{aaaaa5.1})}, the measure $\mathfrak{m}$ satisfies (\ref{eqnfjlajfleajflea}).

{Let us make a comment that by taking $t_2\gg t_1\gg 1$, the ratio
in (\ref{eqnfjlajfleajflea}) can be sufficiently large, which means that the reflection
$f:C_0\to C_0$ does not preserve the measure $\mathfrak{m}$.}
\end{exmp}

{
\begin{exmp}\label{eexmp5.4}
	
	Let $\{g_i\}$ be a sequence of continuous positive functions on $[-1,1]$ such that
    \[
    g_i(x)=\left\{\begin{aligned}
    	\sqrt{i^{-2}-{(x+1-i^{-1})}^2},\ \ & x\in [-1,-1+i^{-1}];\\
    	i^{-2},\ \ \ \ \ \ \ \ \ \ \ \ \ \ \ \ \ \ \ \ \ \ \ \ \ \ &x\in [1/4,3/4];\\
    		\sqrt{i^{-2}-{(x-1+i^{-1})}^2},\ \ & x\in [1,1-i^{-1}],
    \end{aligned}\right. 
\]
and that $g_i$ is smooth on $(-1,1)$ with $\max_{[-1,1]}g_i\leqslant i^{-1}$, $\max_{[-1+i^{-1},1-i^{-1}]}|g_i'|\leqslant J$ and $\sup_{(-1,1)}|(g_i)''|\leqslant J$ for some constant $J>0$.

	For each $i$, let 
	\[
	(N_i,\mathrm{h}_i):=\{(x,y)\in\mathbb{R}^2|\,|y|\leqslant g_i(x)\},
	\]
	where $\mathrm{h}_i:=\mathrm{d} x^2+\mathrm{d} y^2$. Let $(M_i,\mathrm{g}_i)$ be the standard product Riemannian manifold of $(N_i,\mathrm{h}_i)$ and $(\mathbb{S}^1(i^{-1}),{\mathrm{d}t^2})$. It is clear that 
\[
 	\mathrm{Ric}_{\mathrm{g}_i}= 0,\  \mathrm{Ric}_{\mathrm{g}_{\partial M_i}}=0,
\] 
and
\[
|\text{\Rmnum{2}}_{\partial M_i}|\leqslant \frac{|g_i''|}{\left(1+{(g_i')}^2\right)^{3/2}}\leqslant J.
\]

Under the convergence (\ref{ffaeohof}), we have $(X,\mathsf{d}_X)=([-1,1],\mathrm{d}x,f\mathrm{d}x)$, with 
\[
f=\lim_{i\to\infty}\frac{g_i}{\int_{[-1,1]}g_i(x)\mathrm{d}x}.
\] 

From our construction, we know $g_i$ is not convex on $(-1,1)$,  {indicating that $\text{\Rmnum{2}}_{\partial M_i}\geqslant 0$ does not hold for all $i$. Moreover, we have} $f=0$ on $[1/4,3/4]$, which means and the metric measure space $([-1,1],\mathrm{d}x,f\mathrm{d}x)$ does not satisfy the RCD$(K,N)$ condition for any $K\in\mathbb{R}$ and $N\in [1,\infty)$ {because $\mathrm{supp}(f)\neq [-1,1]$}.
\end{exmp}
}

We also make a comment about the least dimension upper bound through the following example.

\begin{exmp}
{{Under the assumption of Theorem \ref{aaathm1.3}, if $\lim_{i\to \infty}\mathrm{vol}_{\mathrm{g}_{\partial M_i}}(\partial M_i)=0$, i.e. }$(C_0,\mathsf{d},\mathfrak{m})$ is a collapsed Ricci limit space, then it follows from \cite[{Theorem 3.1}]{ChCo1} that the Hausdorff dimension of the $(C_0,\mathsf{d})$ is {at most} $n-2$. We will give an example about collapsing limit space of boundaries of inradius collapsed manifolds. The {synthetic upper} dimension bound of this limit space can not be improved to any {$N\in (1,n-1)$}.}

{ Let ${(M_i,\mathrm{g}_i)}$ be the product Riemannian manifold ${M_i:=N_i\times [0,i^{-1}]}$, where 
\[
{(N_i,\mathrm{h}_i)}:=\left\{(x,y,z{,w})\in {\mathbb{R}^4}|\mathop{ x^2+i^2(y^2+z^2{+w^2})=1}\right\},
\]
with $\mathrm{{ h}}_i$ being the natural Riemannian metric induced by $\mathrm{g}_{{\mathbb{R}^4}}$, and $\mathrm{{g}}_i=\mathrm{{h}}_i+{\mathrm{d}t^2}$.


{As $i\rightarrow \infty$, we see $\left(\partial M_i,\mathsf{d}_{\mathrm{g}_{\partial M_i}},\frac{\mathrm{vol}_{\mathrm{g}_{\partial M_i}}}{\mathrm{vol}_{\mathrm{g}_{\partial M_i}}(\partial M_i)}\right)$ $\mathrm{mGH}$ converges to the RCD$(0,{3})$ space $\left([-1,1],\mathrm{d}t,{\frac{3t-t^3+2}{4}\mathrm{d}t} \right)$.
}

Let us recall the one-dimensional specialization of the Bakry–\'{E}mery
CD(K, N) condition for smooth weighted Riemannian manifolds (see \cite{BE85,B94}). For an open interval $I\subset \mathbb{R}$, a function $h\in C^2_{loc}(I)$ is a CD$(K,N)$ density on $I$ (i.e. $(I,\mathrm{d}t,h\mathrm{d}t)$ is a CD$(K,N)$ metric measure space) if and only if 
\[
(N-1)\frac{(h^{\frac{1}{N-1}})''(x)}{h^{\frac{1}{N-1}}(x)}\leqslant -K,\ \forall x\in I.
\]

{Now for $h=\frac{3t-t^3+2}{4}$ and $N\in (1,3)$, we set $\varphi=h^{1/(N-1)}$. Because $\lim_{t\downarrow -1}\varphi''/\varphi=+\infty$, } the upper dimension bound of the metric measure space $([-1,1],\mathrm{d}t,{\frac{3t-t^3+2}{4}}\mathrm{d}t)$ can not be reduced to any $N\in {(1,3)}$.}


\end{exmp}

Finally, {we give an example based on Menguy’s construction \cite[Theorem 0.1]{M00}. This example illustrates the distinction} between the lower sectional curvature bound case (\ref{jfieajifoa}) and the lower Ricci curvature bound case (\ref{1.1}). {For clarity, we adhere to Menguy’s original notation.}
\begin{exmp}
	{	Let $t_i$ be a monotone increasing sequence such that 
		\[
		\lim_{i\to\infty}(t_{i+1}-t_{i})=\infty,\ \lim_{i\to \infty} \frac{r_i}{i}=\infty,
		\] where $r_i=4t_i/\log(\log t_i)$. Let {$M:=[2,\infty)\times\Sigma\,\mathbb{S}^2$, where $\Sigma\,\mathbb{S}^2$ is the suspension over $\mathbb{S}^2$, defined as $[0,\pi]\times\mathbb{S}^2/\sim$, with $\{0\}\times\mathbb{S}^2$ and $\{\pi\}\times\mathbb{S}^2$ identified to south pole $\mathbb{S}^3_-$ and north pole $\mathbb{S}^3_+$ respectively.} The metric $\mathrm{g}$ on $M$ is defined as 
		\[
		\mathrm{g}:=\mathrm{d}t^2+u^2(t)\left(\mathrm{d}x^2+f^2(t,x)\,\mathrm{d}\sigma^2\right).
		\]
	Here $\mathrm{d}\sigma^2$ is the standard metric on $\mathbb{S}^2$. {The function $f:[2,\infty)\times [0,\pi]\to\mathbb{R}$ is defined as 
		\begin{equation}\label{5.7}
		f(t,x)=\left\{\begin{aligned}
			\frac{1}{l(t)}\sin(l(t)x),
			\ \ \ \ \ \ \ \ \ \ \ \ \ \ 
			 \ \ \ \ \ \ \ \ \ \ \ \ \ \ \ \ \ \ \ \ \ \ \ \ \ \  &x\in[0,b(t)],\\
			\frac{1}{l(t)}\sin(l(t)b(t))\exp\left((1-\epsilon(t)\log\left(\frac{x}{b(t)}\right)\right),\ &x\in \left[b(t),\epsilon(t)\right],\\
			R\sin(x+\delta(t)\theta(x)),\ \ \ \ \ \ \ \ \ \ \ \ \ \ \ \ \ \ \ \ \ \ \ \ \ \ \ \ \ \ \ \ \ &x\in \left[\epsilon(t),\frac{\pi}{2}\right]
		\end{aligned}\right.
		\end{equation} such that $f(t,x)=f(t,\pi-x)$, where $R\in (0,1)$ is sufficient small, $\epsilon(t)=(\log^{1/4}t)^{-1}$, $\theta$ is a smooth monotone non-increasing function such that 
		\[
\theta=\left\{\begin{aligned}
			1,\ \ &\text{for}\ x\leqslant 1/2;\\
			0,\ \ &\text{for}\ x\geqslant 1,
		\end{aligned}\right.
		\] and $l,b,\delta$ are smooth positive functions on $(0,\infty)$ defined by \cite[(1.52)-(1.58)]{M00}. The function $u$ is defined as \cite[(1.13), (1.14), (1.19-1.23)]{M00}. Both $f$ and $u$} are piecewise $C^2$ functions with jumps in the second derivative at a discrete set of points, but can be smoothed to $C^2$
		while preserving positive Ricci curvature.

	{Denote by $\Gamma_t:\{t\}\times\Sigma \,\mathbb{S}^2$.} {From our construction we know for every} $t\in (2,\infty)$, the hypersurface $\Gamma_t$ is homeomorphic to $\mathbb{S}^3$, though not necessarily $C^2$-isometric to $\mathbb{S}^3$. Let $X=u^{-1}\partial x$ and $\Sigma_1,\Sigma_2$ be dual vector fields of $u(t)f(t,x)d\sigma$. For every hypersurface $\Gamma_t$, { we denote by $\text{\Rmnum{2}}_t$} the second fundamental form of $\Gamma_t$. {With respect to the orthonormal basis 
		 $\{X,\Sigma_1,\Sigma_2\}$, an application of Koszul formula guarantees that $\text{\Rmnum{2}}_t$ has the following expression.} 
	\[
		\text{\Rmnum{2}}_t:=
			\left(\begin{array}{lll}
				\frac{u_t}{u}&0&0\\
				0&\frac{u_t}{u}+\frac{f_t}{f}&0\\
				0&0&\frac{u_t}{u}+\frac{f_t}{f}\\	\end{array}\right).
	\]
	By \cite[(1.30), (1.69)]{M00}, for sufficiently large $t$, $\text{\Rmnum{2}}_t$ has a two sided bound $O(1/t)$. {Owing to \cite[(1.168), (1.169)]{M00},} the intrinsic sectional curvature $\mathrm{Sec}_t$ of $\Gamma_t$ satisfies
	\[
	\mathrm{Sec}_t(\Sigma_1\wedge X)=-\frac{f_{xx}}{u^2 f},\ \mathrm{Sec}_t(\Sigma_1\wedge \Sigma_2)=\frac{1-{f_x}^2}{u^2 f^2},
	\] 
	both of which are positive by \cite[(1.64)-(1.66), (1.114)]{M00}. Moreover, by \cite[(1.107)-(1.116)]{M00}, the Ricci curvature of $(M,\mathrm{g})$ is positive everywhere.

	For each $i$ and $j=0,\ldots,i-1$, by \cite[Section 1.6]{M00},  $B_{2^{-i}r_i}((t_i+jr_i/i)\times{\mathbb{S}^3_-})$ satisfies the Perelman's property \cite[Definition 0.13]{M00}. This allows us to remove each such ball and glue in a copy of $\mathbb{C}\mathrm{P}^2$ along its boundary within the annulus  \[
	{([t_i-2r_i,t_i+2r_i]\times \Sigma\, \mathbb{S}^2,\mathrm{g}).}
	\]
	After smoothing, the resulting Riemannian manifold, which is originally the above annulus, is denoted by $(N_i,\mathrm{h}_i)$.

	According to \cite[(1.28), (1.29)]{M00}, there exists a constant $J>1$ such that for all sufficiently large $t$ it holds $Jt\geqslant u(t)\geqslant J^{-1}t$. Define the rescaled manifold $(M_i,\mathrm{g_i}):=(N_i,{t_i}^{-2}\mathrm{h}_i)$. Then we have $(M_i,\mathrm{g_i})\in \mathcal{M}(4,0,0,J,J)$, and \[
	{ 0<\liminf_{i\to\infty}\mathrm{diam}(M_i)\leqslant \limsup_{i\rightarrow \infty}\mathrm{diam}(M_i)<\infty},\ \lim_{i\to \infty}\mathrm{inrad}(M_i)=0.
	\] However, {it is obvious} that $M_i$ is not diffeomorphic to $\mathbb{S}^3\times [0,1]$. {After passing to a subsequence, assume $u(t_i)/t_i\to u_\infty\in [J^{-1},J]$ as $i\to\infty$. Then a combination of (\ref{5.7}) and \cite[(1.52)]{M00} implies that $\lim_{t\to\infty}\delta(t)=0$ and $\{(M_i,\mathsf{d}_{\mathrm{g}_i})\}$ Gromov-Hausdorff converges to $(\Sigma \,\mathbb{S}^2,\mathsf{d}_{\mathrm{g}_\infty})$, where the metric $\mathrm{g}_\infty$ is defined as
	\[
	\mathrm{g}_\infty:=u_\infty \left(\mathrm{d}x^2+R^2\sin^2 x\,\mathrm{d}\sigma^2\right).
	\]This space is smooth except for singularities at the two poles.} 
	}

\end{exmp}

%
%





\bibliographystyle{plain}
\bibliography{ref}

\bigskip


\end{document}